\documentclass{amsart}

\usepackage{amsmath}
\usepackage{amssymb}
\usepackage{amscd}

\makeatletter
\newtheorem*{rep@theorem}{\rep@title}
\newcommand{\newreptheorem}[2]{%
\newenvironment{rep#1}[1]{%
 \def\rep@title{#2 \ref{##1}}%
 \begin{rep@theorem}}%
 {\end{rep@theorem}}}
\makeatother

\title[$L^2$-Betti numbers and equivariant approximation]
 {Characters, $L^2$-Betti numbers and an equivariant approximation theorem}

\author[S. Kionke]{Steffen Kionke}

\address{Steffen Kionke\\
         Mathematisches Institut\\
         Heinrich-Heine-Universit\"at\\
         Universit\"ats\-str.~1\\
         40225 D\"usseldorf\\ Germany}
\email{steffen.kionke@uni-duesseldorf.de}

\theoremstyle{plain}
\newtheorem{theorem}{Theorem}
\newtheorem{lemma}[theorem]{Lemma}
\newtheorem{corollary}[theorem]{Corollary}
\newtheorem{proposition}[theorem]{Proposition}
\newreptheorem{theorem}{Theorem}
\newreptheorem{proposition}{Proposition}

\theoremstyle{definition}
\newtheorem{definition}[theorem]{Definition}
\newtheorem{deflem}[theorem]{Definition \& Lemma}
\newtheorem{remark}[theorem]{Remark}
\newtheorem{example}[theorem]{Example}
\newtheorem{question}[theorem]{Question}

\numberwithin{equation}{section}
\numberwithin{theorem}{section}

\DeclareMathOperator{\id}{Id}
\DeclareMathOperator{\im}{im}
\DeclareMathOperator{\ind}{ind}
\DeclareMathOperator{\Ind}{Ind}
\DeclareMathOperator{\Tr}{Tr}
\DeclareMathOperator{\Hom}{Hom}
\DeclareMathOperator{\End}{End}
\DeclareMathOperator{\Aut}{Aut}

\DeclareMathOperator{\fix}{Fix}
\DeclareMathOperator{\res}{res}
\DeclareMathOperator{\Res}{Res}

\DeclareMathOperator{\Ch}{Ch}
\DeclareMathOperator{\rnk}{rk}
\DeclareMathOperator{\nulli}{null}
\DeclareMathOperator{\infl}{infl}
\DeclareMathOperator{\pos}{Pos}
\DeclareMathOperator{\Perm}{Per}
\DeclareMathOperator{\Irr}{Irr}

\DeclareMathOperator{\adeg}{a-deg}
\DeclareMathOperator{\ord}{ord}

\providecommand{\aCh}{\Ch_{f}}
\providecommand{\normal}{\trianglelefteq}

\providecommand{\calH}{\mathcal{H}}
\providecommand{\calN}{\mathcal{N}}
\providecommand{\calM}{\mathcal{M}}
\providecommand{\calB}{\mathcal{B}}
\providecommand{\calO}{\mathcal{O}}

\providecommand{\norm}[1]{\lVert #1 \rVert}

\providecommand{\triv}{1\!\!1}

\providecommand{\bbN}{\mathbb{N}}
\providecommand{\bbR}{\mathbb{R}}
\providecommand{\bbQ}{\mathbb{Q}}
\providecommand{\bbZ}{\mathbb{Z}}

\providecommand{\bbC}{\mathbb{C}}

\DeclareMathOperator{\SL}{SL}
\DeclareMathOperator{\GL}{GL}
\DeclareMathOperator{\Gal}{Gal}
\DeclareMathOperator{\U}{U}
\begin{document}

\begin{abstract}
Let $G$ be a group with a finite subgroup $H$. We define the $L^2$-multiplicity
of an irreducible representation of $H$ in the $L^2$-homology of a proper $G$-CW-complex.
These invariants generalize the $L^2$-Betti numbers.
Our main results are approximation theorems for $L^2$-multiplicities which extend the approximation theorems for $L^2$-Betti numbers of
L\"uck, Farber and Elek-Szab\'o respectively.
The main ingredient is the theory of characters of infinite groups and a method to induce characters from finite subgroups.
We discuss applications to the cohomology of (arithmetic) groups.
\end{abstract}

\maketitle

\section{Introduction}

The philosophy of $L^2$-invariants is to take a classical homological invariant of finite simplicial complexes and to mimic its construction
to define an $L^2$-analog of this invariant for simplicial complexes with a free cocompact action of a group~$G$.
From the viewpoint of homological algebra this amounts to replacing chain complexes of $\bbZ$-modules by chain complexes of modules over the tremendously complicated
group ring $\bbZ[G]$.

The most famous $L^2$-invariants are the $L^2$-Betti numbers. To define them one embeds the group ring $\bbZ[G]$ into the group von Neumann algebra $\calN(G)$ and 
uses the von Neumann dimension to define the dimension of kernel and image of the boundary operators.
In particular, given a free cocompact $G$-CW-complex $X$, one defines the $L^2$-Betti numbers $b_p^{(2)}(X;G)$ 
by applying this method to the cellular chain complex of $X$. 
Without doubt, it takes some effort to work out the details and we refer the reader to \cite{LuckBook} for a detailed account. 

\medskip

In this article we define and study $L^2$-invariants which refine the $L^2$-Betti numbers.
We first describe the finite dimensional analog.
Consider a finite simplicial complex $X$ with an action of a finite group $H$. The action is not assumed to be free.
The group $H$ acts on the $p$-th homology $H_p(X,\bbC)$ of $X$. This yields a finite dimensional complex representation 
of $H$ and every such representation is a direct sum of irreducible ones.
In particular, we may consider the \emph{multiplicity} $m(\chi, H_p(X,\bbC))$ of an irreducible representation $\chi \in \Irr(H)$ in $H_p(X,\bbC)$.
%The multiplicities refine the $p$-th Betti number, since $b_p(X) = \sum_{\chi\in \Irr(H)} \chi(1) m(\chi, H_p(X,\bbC))$.

A major part of this article is devoted to defining corresponding $L^2$-multiplicities. 
Given a group $G$ with a finite subgroup $H$ and a proper cocompact $G$-CW-complex $X$, we obtain the $p$-th
\emph{$L^2$-multiplicity} $m^{(2)}_p(\chi, X; G)$ for every irreducible representation $\chi \in \Irr(H)$.
Suffice it to say that the $L^2$-multiplicities are non-negative real numbers defined as suitably normalized von Neumann dimensions with respect to 
certain finite von Neumann algebras.
We postpone the concise definition to Section \ref{sec:mainResults} as it requires some more terminology.

\medskip

The strength of $L^2$-invariants stems on the fact that they can often be approximated by their finite dimensional analogues.
The prototype of this phenomenon is L\"uck's approximation theorem.
\begin{theorem}[L\"uck \cite{Luck1994}]
  Let $X$ be a free cocompact $G$-CW-complex. Let $(\Gamma_n)_{n\in\bbN}$
  be a decreasing sequence of finite index normal subgroups in $G$ with $\bigcap_{n\in\bbN} \Gamma_n = \{1\}$.
  Then
  \begin{equation*}
     \lim_{n\to\infty} \frac{b_p(X/\Gamma_n)}{[G:\Gamma_n]} = b_p^{(2)}(X;G).
  \end{equation*}
\end{theorem}
For several other instances of approximation we refer the reader to the survey~\cite{LuckSurvey2016}.
Our main results concern the approximation of $L^2$-multiplicities by finite dimensional multiplicities.
The simplest version can be stated as follows.
\begin{theorem}\label{thm:approximation1}
 Let $G$ be a group and let $H \leq G$ be a finite subgroup.
 Let $X$ be a proper cocompact $G$-CW-complex and let $(\Gamma_n)_{n\in\bbN}$
  be a decreasing sequence of finite index normal subgroups in $G$ with $\bigcap_{n\in\bbN} \Gamma_n = \{1\}$.
  Then
  \begin{equation*}
     \lim_{n\to\infty} \frac{m(\chi, H_p(X/\Gamma_n,\bbC))}{[G:\Gamma_n]} = m_p^{(2)}(\chi, X;G)
  \end{equation*}
  for every irreducible representation $\chi \in \Irr(H)$.
\end{theorem}

L\"uck's approximation theorem has been generalized in several ways and we will obtain more general versions of Theorem \ref{thm:approximation1} as well.
Firstly, Farber \cite{Farber1998} extended L\"uck's theorem by replacing the assumption
that the subgroups $\Gamma_n$ be normal by the weaker \emph{Farber condition}.
We will introduce an equivariant Farber condition and prove a generalization for $L^2$-multiplicities in Theorem \ref{thm:eqFarber}.
Secondly, Elek-Szab\'o \cite{ElekSzabo2005} obtained an approximation theorem for $L^2$-Betti numbers by replacing 
the assumption that the groups $\Gamma_n$ be of finite index by the condition that every quotient group $G/\Gamma_n$ is sofic.
A similar result for $L^2$-multiplicities is Theorem \ref{thm:eqSofic} below. It is based on the new notion of relative soficity.
We will not discuss approximation for sequences of lattices in locally compact groups as in \cite{ABBGNRS, PST2016}.

\medskip

In some cases the $L^2$-multiplicities can be computed from the $L^2$-Betti numbers.
\begin{repproposition}{prop:centralizer}
  If every element $h \in H\setminus\{1\}$ satisfies $[G:C_G(h)] = \infty$, then
  $m_p^{(2)}(\chi, X;G) = \frac{\dim(\chi)}{|H|} \: b_p^{(2)}(X; G)$ for every $\chi \in \Irr(H)$.
\end{repproposition}
At first glance this seems disappointing, as it means that the $L^2$-multiplicities (in this case) are not really new invariants.
However, even in this case Theorem \ref{thm:approximation1} is stronger than L\"uck's approximation theorem. 
In addition, the proposition makes it possible to compute the $L^2$-multiplicities in many cases.
This leads to new applications of $L^2$-Betti numbers. For instance, Theorem \ref{thm:approximation1} can be used to show that every irreducible
representation occurs with arbitrarily large multiplicity in the homology of a large finite covering space.
Here is a version of this result for the cohomology of groups.

\begin{reptheorem}{thm:groupCohomology}
   Fix $p\in \bbN$.
   Let $\Gamma$ be a group of type $F_{p+1}$ and let $H$ be a finite group which acts on $\Gamma$ by automorphisms.
   Assume that the fixed point group $\Gamma^h$ has infinite index in $\Gamma$ for all $h\in H\setminus\{1\}$.
   
   Let $\Gamma \supset \Gamma_1 \supset \Gamma_2 \supset \dots$ be a decreasing sequence of finite index normal $H$-stable subgroups with $\bigcap_{n}\Gamma_n = \{1\}$.
    Then
   \begin{equation*}
       \lim_{n\to \infty} \frac{m(\chi, H^p(\Gamma_n,\bbC))}{[\Gamma:\Gamma_n]} = \frac{\dim(\chi)}{|H|} \: b^{(2)}_p(\Gamma) 
   \end{equation*}
   for every irreducible representation $\chi \in \Irr(H)$.
\end{reptheorem}

\begin{corollary}
  Let $\Gamma$ be a residually finite group of type $F_{p+1}$ with $b^{(2)}_p(\Gamma) \neq 0$.
  Suppose that $H$ is a finite group which acts on $\Gamma$ by automorphisms such that
  $[\Gamma:\Gamma^h] = \infty$ for all $h\in H\setminus\{1\}$.
  For every $\chi \in \Irr(H)$ and every $B\in \bbN$, there is a finite index $H$-stable subgroup $\Gamma_0 \leq \Gamma$
  such that $m(\chi, H^p(\Gamma_0,\bbC)) \geq B$.
\end{corollary}

The reader is encouraged to experiment with this theorem to see that it yields non-obvious results. 
In particular, we may apply this to arithmetic groups. The following is a special case of Corollary \ref{cor:GaloisAction}.
\begin{corollary}\label{cor:introSL2}
 Let $F/\bbQ$ be a totally real Galois extension of degree $d = [F:\bbQ]$.
 Let $\Gamma \subseteq \SL_2(F)$ be a $\Gal(F/\bbQ)$-stable arithmetic subgroup.
 Every irreducible representation of the Galois group $\Gal(F/\bbQ)$ occurs with arbitrarily large 
 multiplicity in the cohomology $H^d(\Gamma_0,\bbC)$ 
 of a sufficiently small subgroup $\Gamma_0 \leq_{f.i.} \Gamma$.
\end{corollary}
It is worth noting, that this corollary can also be obtained without $L^2$-methods using Lefschetz numbers instead.
Actually, the observation of this result lead us to discover Theorem \ref{thm:approximation1}.

\subsection*{Methods}
The approach in this article is based on the space of characters $\Ch(G)$ of the group $G$.
Every such character can be used to define a suitable notion of Betti numbers. This is probably known to the experts. Nevertheless,
we decided to include this material and we tried to present it in a way which
resembles closely the definition $L^2$-Betti numbers (circumventing some more advanced notions from the theory of von Neumann algebras). 
Moreover, we made an effort to keep the presentation to a certain degree self-contained.

The main ingredient in order to define the $L^2$-multiplicities is the construction of
\emph{characters induced from finite subgroups}. The $L^2$-multiplicity of a character $\chi$ is defined to be the Betti number of the
character induced from $\chi$. The idea behind this is the well-known Frobenius reciprocity in the representation theory of finite groups.

Finally, the approximation theorem still relies
on the same ideas as the theorems of L\"uck and Farber: 
weak convergence of spectral measures and a ``L\"uck Lemma'' to bound the Fuglede-Kadison determinant.
We discuss approximation in some generality using convergent sequences of characters in $\Ch(G)$.

\subsection*{Organization of the article} The article consists of four sections.

\noindent\ref{sec:CharactersAndBettiNumbers}. \emph{Characters and Betti numbers.}
We introduce Hilbert $G$-bimodules and explain how they can be used to define dimensions and Betti numbers.
We observe that Hilbert $G$-bimodules correspond to characters of the group and we study the space of characters.
The main result is the continuity of the spectral measure map.

\noindent\ref{sec:soficCharacters}. \emph{Sofic and arithmetically hyperlinear characters.} 
We define \emph{sofic} and \emph{arithmetically hyperlinear} characters.
These are characters which can be approximated (in a suitable sense) by finite permutation representations and representations with finite image respectively.
We show that convergent sequences of these characters yield convergent sequences of Betti numbers.

\noindent\ref{sec:Induction}. \emph{Induction of characters from finite groups}.
We define induction of characters via characters. A special case is the induction of characters from finite subgroups.
We show that this construction preserves \emph{arithmetic hyperlinearity} of characters under some conditions. This leads us
to the definition of groups which are sofic relative to a finite subgroup.

\noindent\ref{sec:mainResults}. \emph{Main results and applications.}
We prove and state the main results of this article: the ``sofic'' and the ``Farber'' approximation theorem for $L^2$-multiplicities.
The results stated in the introduction are deduced.
We discuss applications to the cohomology of groups and (as an example) consider the Galois action in the cohomology of arithmetic groups.

\subsection*{Acknowledgements}
The research described in this article was conducted at the Hausdorff Research Institute for Mathematics in Bonn during the trimester program ``Topology'' in 2016. I would like to thank the institute 
for the hospitality and support. Moreover, I would like to thank the members of the research group \emph{$L^2$-Invariants}, namely Gerrit Hermann, Holger Kammeyer, Aditi Kar and Jean Raimbault
for the stimulating discussions during the program.

\section{Characters and Betti numbers}\label{sec:CharactersAndBettiNumbers}

Throughout $G$ denotes a discrete group. 

\subsection{Tracial Hilbert $G$-bimodules}
\begin{definition}
A \emph{tracial Hilbert $G$-bimodule} is a quadruple $(\calH, \ell, r, e)$ consisting of 
a Hilbert space $\calH$, a homomorphism $\ell\colon G \to \U(\calH)$, an antihomomorphism $r\colon G \to \U(\calH)$,
and a vector $e\in\calH$, such that
\begin{enumerate}
   \item\label{it:bimodule1} the actions commute, i.e.\ $\ell(g)\circ r(h) = r(h) \circ \ell(g)$ for all $g,h\in G$,
   \item\label{it:bimodule2} $\ell(g)e = r(g)e$ for all $g \in G$, and
   \item\label{it:bimodule3} $\norm{e}=1$ and the translates $(\ell(g)e)_{g\in G}$ span a dense subspace of $\calH$.
\end{enumerate}
We write $\ell(g)x = g \cdot x$ and $r(g)x = x\cdot g$ for all $x\in \calH$. For brevity, we often denote a tracial Hilbert $G$-bimodule only by ($\calH$,$e$) or $\calH$.
\end{definition}

A tracial Hilbert $G$-bimodule $(\calH, \ell, r, e)$ always has a unique anti-linear involution $J\in U(\calH)$ such that 
$J\circ \ell(g) = r(g^{-1}) \circ J$ and $J(e) = e$. Indeed, one can just define $J(g\cdot e) = e \cdot g^{-1}$ and extend it anti-linearly and continuously to $\calH$. Moreover, we can associate to $\calH$ a pair of von Neumann algebras defined as the commutants of the left and right action. We will consider $\ell(G)'$ as acting on $\calH$ from the right and $r(G)'$ as acting from the left.
One can check that $\ell(G)' = r(G)''$ and $r(G)' = \ell(G)''$ and moreover, $J\ell(G)'J = r(G)'$.
In particular, we denote the commutant algebra of the right action by  $\calN_\calH = \ell(G)'' =  r(G)'$.

The vector $e$ is a trace-element for $\calN_\calH$ (in the sense of \cite[I.6.\ Def.~3]{DixmiervNA}). Indeed,
for all $g,h \in G$ we have
\begin{equation*}
 \langle gh\cdot e, e \rangle = \langle g\cdot e \cdot h, e \rangle = \langle g\cdot e , e \cdot h^{-1}\rangle
 = \langle g\cdot e , h^{-1} \cdot e \rangle
 = \langle hg\cdot e, e \rangle.
\end{equation*}
The operators spanned by the left multiplication operators are dense in $\calN_\calH$ with respect to the weak topology,
and thus $\langle ab\cdot e, e\rangle = \langle ba \cdot e, e \rangle$ for all $a,b \in \calN_\calH$.
This means, the von Neumann algebra $\calN_\calH$ has a finite faithful normal trace $\tau$ defined as
\begin{equation*}
   \tau(a) = \langle a\cdot e, e \rangle.
\end{equation*}
The trace defines an inner product on $\calN_\calH$ by the formula $\langle a,b\rangle=\tau(b^*a)$.
It is easy to see that $\calH$ is isometrically isomorphic to the completion $L^2(\calN_\calH,\tau)$ of the von Neumann algebra with respect to this inner product.

\begin{example}
 Let $L^2(G)$ denote the square summable functions on $G$.
 Clearly, it is a Hilbert space with a unitary left and right action of $G$. Let $e$ denote the characteristic function of the identity element $1_G$. Then $(L^2(G),e)$ is a Hilbert $G$-bimodule and the associated von Neumann algebra is the group von Neumann algebra $\calN(G)$ with the usual von Neumann trace.
\end{example}

As for the group von Neumann algebra, it is possible to define Hilbert $\calN_\calH$-modules which have a suitable dimension theory.
\begin{definition}
   Let $(\calH,e)$ be tracial Hilbert $G$-bimodule. A \emph{finitely generated Hilbert $\calN_\calH$-module} is 
   a Hilbert space $V$ with a unitary right $G$-action which admits a right $G$-equivariant and isometric embedding into $\calH^n$ for some $n \in \bbN$.
   A \emph{homomorphism} of Hilbert $\calN_\calH$-modules is a bounded $G$-equivariant linear operator.
\end{definition}

Given a f.g.\ Hilbert $\calN_\calH$-module $V$ and an endomorphism $f\colon V\to V$, then we define 
the von Neumann trace of $f$ by the following usual steps.
Choose an embedding of $\iota\colon V\to \calH^n$ and
let $P$ denote the orthogonal projection onto $\iota(V)$.
The operator $F = \iota \circ f \circ P$ is $r(G)$-equivariant, thus
is given by the multiplication with a matrix $(a_{i,j})_{i,j=1}^n \in M_n(\calN_\calH)$.
The trace of $f$ on $V$ is defined to be
\begin{equation*}
   \Tr_{\calN_\calH}(f) := \sum_{i=1}^n \tau(a_{i,i}) = \sum_{i=1}^n \langle a_{i,i} \cdot e, e \rangle.
\end{equation*}
 As for the group von Neumann algebra one can check that this definition is independent of the chosen embedding.
 The \emph{von Neumann dimension} of a f.g.\ Hilbert $\calN_\calH$-module is the trace is the identity map on $V$, i.e.\
 \begin{equation*}
      \dim_{\calN_\calH}(V) = \Tr_{\calN_\calH}(\id_V).
 \end{equation*}
 
 For later use we remind the reader of the following lemma which is a consequence of the polar decomposition of bounded operators on Hilbert spaces.
 \begin{lemma}\label{lem:homomorphiesatz}
 Let $T\colon V \to W$ be a homomorphism of f.g.\ Hilbert 
 $\calN_\calH$-modules.
 Then $V \cong \ker(T) \oplus \overline{\im(T)}$  as Hilbert $\calN_\calH$-modules.
 \end{lemma}
 The proof is the same as the proof for group von Neumann algebras discussed in Section 1.1.2 in~\cite{LuckBook}.

 \subsection{Characters}
\begin{definition}\label{def:characters}
A function $\phi \in L^\infty(G)$ is of \emph{positive type} if
the matrix $(\phi(g^{-1}_i g_j))_{i,j}$ is positive semi-definite for every finite list of elements $g_1,g_2,\dots,g_n \in G$.
A function $\phi \in L^\infty(G)$ of positive type is called a \emph{character} if it is constant on conjugacy classes and $\phi(1_G)=1$.
\end{definition}

Given a tracial Hilbert $G$-bimodule $(\calH,e)$, then the trace $\tau$ restricted to $\ell(G)$ defines a character
\begin{equation*}
    \tau(g) = \langle g\cdot e, e\rangle.
\end{equation*}
An important observation is that every character arises in this way and the character determines the tracial Hilbert $G$-bimodule up to isomorphism.
This is a consequence of the GNS-construction and goes back Thoma.

\begin{proposition}[cf.\ Lemma 1 in \cite{Thoma1964}]\label{prop:Thoma}
The map which associates to a tracial Hilbert $G$-bimodule $(\calH,e)$ its character $\tau$ induces a bijective correspondence between the isomorphism classes of tracial Hilbert $G$-bimodules and the characters of $G$.
\end{proposition}

Given a character $\phi$ of $G$ we will denote the associated tracial Hilbert $G$-bimodule by $\calH_\phi$ and the corresponding finite von Neumann algebra
by $\calN_\phi = \calN_{\calH_\phi}$.

\begin{example}
There are various ways to construct characters of a group $G$.
\begin{itemize}
 \item The \emph{regular} character $\delta_G^{(2)} = \delta^{(2)}$ of $G$ is the character of $L^2(G)$, this is,
 $\delta^{(2)}(g)=0$ for all $g\neq 1_G$. 
 \item The \emph{trivial} character $\triv$ of $G$ is defined as
 $\triv(g)=1$ for all $g \in G$. It corresponds to the one-dimensional bimodule with trivial action.
 
\item Let $\rho\colon G \to \U(n)$ be a finite dimensional unitary representation of $G$, then the normalized trace
    $\phi(g) = \frac{1}{n}\Tr(\rho(g))$ is a character of $G$. This special case will be discussed in more generality in Section \ref{sec:finitedimensional} below.
%     Assume that the representation $\rho$ is irreducible.
%     Then the corresponding tracial Hilbert $G$-bimodule is the space $M_n(\bbC)$ of 
%     complex $n\times n$-matrices with left and right multiplication by~$\rho(G)$. 
%     The inner product on $M_n(\bbC)$ is the trace form $\langle a,b\rangle = \Tr(b^*a)$ and the tracial vector is the identity matrix.
\item Let $(X,\mu)$ be a standard Borel probability space with a measure preserving action of $G$. Measuring the fixed point sets of the elements of $G$ defines a character,
i.e.\ $\phi(g) = \mu(\fix_X(g) )$.
\end{itemize}
\end{example}
\begin{remark}
 There are infinite groups which have only the trivial and the regular character. For instance, the commutator subgroups of Higman-Thompson groups have this property (see \cite{DudkoMedynets}).
\end{remark}

\subsection{Betti numbers of chain complexes of finitely generated projective $\bbC[G]$-modules}\label{sec:BettiNumbers}
 We will use tracial Hilbert $G$-bimodules (or characters) mainly as a way to associate a rank to a homomorphism between 
 finitely generated projective $\bbC[G]$-modules and to define Betti numbers. We briefly describe how this works. 
 Here modules are always considered to be right modules.
 
 In the following we will consider finitely generated projective modules over $\bbQ[G]$ and $\bbC[G]$.
 The main examples are described in the following lemma.
 \begin{lemma}\label{lem:finitestabilizersProjective}
   Let $F$ be a field of characteristic $0$.
   If $S \leq G$ is a finite subgroup,
   then the permutation module $F[S\backslash G]$ is a finitely generated projective right $F[G]$-module.
 \end{lemma}
 \begin{proof}
    Let $p\colon F[G] \to F[S\backslash G]$ be the canonical projection.
    The short exact sequence of right $F[G]$-modules
    \begin{equation*}
       0 \:\longrightarrow \: \ker(p) \:\longrightarrow \:F[G] \:\stackrel{p}{\longrightarrow} \:F[S\backslash G] \:\longrightarrow \: 0
    \end{equation*}
    is split via the morphism $F[S\backslash G] \to F[G]$ which send $Sg$ to $\frac{1}{|S|}\sum_{s\in S} sg$.
 \end{proof}

 Let $Q$ be a f.g.\ projective $\bbC[G]$-module
 Let $\phi$ be a character of $G$ and let $\calH_\phi$ be the associated tracial Hilbert $G$-bimodule.
 Then $Q\otimes_{\bbC[G]}\calH_\phi$ is a finitely generated Hilbert $\calN_\phi$-module.
 Indeed, let $Q\oplus Q' \cong \bbC[G]^n$, then we can identify $Q\otimes_{\bbC[G]}\calH_\phi$
  with a closed subspace of $\calH_\phi^n = \bbC[G]^n \otimes_{\bbC[G]} \calH_\phi$.
 
 Let $f\colon Q_1 \to Q_2$ be a homomorphism between f.g.\ projective $\bbC[G]$-modules.
 We obtain a bounded linear $G$-equivariant operator 
 \begin{equation*}
 f_{\phi} = f\otimes\id\colon Q_1\otimes_{\bbC[G]}\calH_\phi \to Q_2\otimes_{\bbC[G]}\calH_\phi.
 \end{equation*}
 Its kernel $\ker(f_\phi)$ and the closure of its image $\overline{\im(f_\phi)}$ are 
 f.g.\ Hilbert $\calN_\phi$-modules. In particular, we can consider the von Neumann dimensions
 \begin{align*}
     \nulli_\phi(f) := \dim_{\calN_\phi} \ker(f_\phi),\\
     \rnk_\phi(f) := \dim_{\calN_\phi} \overline{\im(f_\phi)},
 \end{align*}
 which we will call the $\phi$-nullity and $\phi$-rank of $f$ respectively.
 In particular, the $\phi$-nullity and $\phi$-rank of matrix $A \in M_{m,n}(\bbC[G])$ are defined
 as the $\phi$-nullity and $\phi$-rank of the left multiplication map $l_A\colon \bbC[G]^n \to \bbC[G]^m$.
 In this case Lemma \ref{lem:homomorphiesatz} yields the useful identity
 \begin{equation}\label{eq:dim-formula}
      \rnk_\phi(A) + \nulli_\phi(A) = n
 \end{equation}
 for every character $\phi$ of $G$.
 The next lemma allows us to restrict to the consideration of matrices in many situations.
 
 \begin{lemma}\label{lem:existenceOfMatrices}
 Let $R\subseteq \bbC$ be a subring.
 Let $f\colon Q_1\to Q_2$ be a homomorphism of f.g.\ projective $R[G]$-modules.
 There are matrices $A$ and $B$ over the group ring $R[G]$
 such that $\rnk_\phi(f) = \rnk_\phi(A)$ and $\nulli_\phi(f) = \nulli_\phi(B)$ for all characters $\phi$.
 \end{lemma}
 \begin{proof}
 Take f.g.\ projective $R[G]$-modules $Q'_1$ and $Q'_2$ such that
 $Q_1\oplus Q'_1 \cong R[G]^n$ and $Q_2 \oplus Q'_2 \cong R[G]^m$ for some integers $n$ and $m$.
 Define $\alpha\colon Q_1\oplus Q'_1 \to Q_2 \oplus Q'_2$ by $\alpha(x,y) = (f(x),0)$.
 Then $\overline{\im(f_\phi)} \cong \overline{\im(\alpha_\phi)}$ as Hilbert $\calN_\phi$-modules for every character~$\phi$.
 Take $A$ to be the matrix of $\alpha$ with respect to some basis. 
 The statement for the kernel follows analogously.
 \end{proof}
 
 \begin{definition}\label{def:phi-Bettinumber}
 Let
 \begin{equation*}
     Q_\bullet\colon \dots \longrightarrow Q_{n+1} \stackrel{\partial_{n+1}}{\longrightarrow} Q_{n} \stackrel{\partial_{n}}{\longrightarrow} Q_{n-1} \longrightarrow \dots
 \end{equation*}
  be a chain complex of $\bbC[G]$-modules and assume that each $Q_i$ is finitely generated and projective.
  We define the $n$-th $\phi$-\emph{Betti number} of $Q_\bullet$ to be
  \begin{equation*}
       b^\phi_n(Q_\bullet) = \nulli_\phi(\partial_{n}) - \rnk_\phi(\partial_{n+1}).
  \end{equation*}
  \end{definition}
 The objective of our investigations is to understand how the rank, nullity and Betti numbers vary with the 
 character $\phi$.

\subsection{The space of all characters}

Let $\Ch(G)$ denote the set of all characters of $G$. With the topology of pointwise convergence $\Ch(G)$ is a compact subset of the space of class functions on $G$. Here we briefly discuss the structure of this space.

The character space $\Ch(G)$ is convex. The extremal points are called \emph{indecomposable characters}. 
In fact, a character $\phi$ is indecomposable exactly if the corresponding von Neumann algebra $\calN_\phi$ is a factor. 
We note that the pointwise product $(\phi\psi)(g) = \phi(g)\psi(g)$ of characters $\phi,\psi\in\Ch(G)$ is again a character.

A map between character spaces which preserves convex combinations will be called \emph{affine}. A map which preserves products of characters 
is called \emph{multiplicative}.
Let $\alpha\colon G \to H$ be a homomorphism of groups.
We can pullback $\phi \in \Ch(H)$ along $\alpha$ to obtain a character $\alpha^*(\phi)$ of $G$, i.e.\ $\alpha^*(\phi)(g) = \phi(\alpha(g))$. 
The pullback of characters defines an affine, multiplicative, continuous map $\alpha^*\colon \Ch(H) \to \Ch(G)$ between character spaces.
There are two important special cases of this construction.
(1) Suppose that $G$ is a subgroup of $H$ and $\alpha$ is the inclusion homomorphism, then the pullback is called the \emph{restriction map}
$\res^H_G := \alpha^*$. (2) Suppose $\alpha\colon G \to H$ is a quotient homomorphism, then the pullback map is called the 
\emph{inflation} $\infl^G_{H} := \alpha^*$. 
 
 \subsection{A norm on matrices over $\bbC[G]$}
 In this section we introduce a norm on the $m\times n$-matrices $M_{m,n}(\bbC[G])$ over the group ring.
 \begin{lemma}[cf.\ 2.5 in \cite{Luck1994}]\label{lem:universalbound}
  For every matrix $A \in M_{m,n}(\bbC[G])$ there is a constant $c_A > 0$ such that for every
  character $\phi\in\Ch(G)$
  \begin{equation*}
       \norm{l_A}_\phi \leq c_A
  \end{equation*}
  where $\norm{l_A}_\phi$ denotes the operator norm of the multiplication operator $l_A$ on $\calH_\phi^n$.
 \end{lemma}
 \begin{proof}
 Write $A = \sum_{i=1}^r b_i U_i$ with $b_i \in \bbC$ and where the $U_i$ are matrices with precisely one non-zero entry of the form $g\in G$.
 Then $c_A = \sum_{i=1}^r |b_i|$ is as required, since the $U_i$ act as partial isometries on every Hilbert $G$-bimodule.
 \end{proof}
 
 \begin{deflem}\label{lem:charactersup-norm}
 For $A \in M_{m,n}(\bbC[G])$ the formula
 \begin{equation*}
    \norm{A}_\infty := \sup_{\phi \in \Ch(G)} \norm{l_A}_{\phi}
 \end{equation*}
 defines a norm on $M_{m,n}(\bbC[G])$, which will be called the \emph{character sup-norm}.
 The character sup-norm is submultiplicative, i.e.\ 
 for all $A \in M_{m,n}(\bbC[G])$ and $B \in M_{n,k}(\bbC[G])$ we have
 \begin{equation*}
        \norm{AB}_\infty \leq \norm{A}_\infty \norm{B}_\infty.
 \end{equation*}
 \end{deflem}
  \begin{proof}
   By Lemma \ref{lem:universalbound} $\norm{A}_\infty$ is finite.
   Hence, it is clear that the formula defines a semi-norm. Moreover, for the regular character $\delta^{(2)}$
   the induced map from the group algebra $\bbC[G]$ into the group von Neumann algebra $\calN(G)$ is injective,
   thus $\norm{l_A}_{\delta^{(2)}}=0$ exactly if $A = 0$. We conclude that the character sup-norm is indeed a norm.
   The submultiplicativity is immediate from the fact that the operator norms are submultiplicative.
 \end{proof}
 
 Let $\pos_n(\bbC[G])$ denote the set of self-adjoint $(n\times n)$-matrices over the group algebra $\bbC[G]$
 such that the left multiplication operator $l_{A}\in \calB(\calH_\phi^n)$ is positive for every character $\phi$.
 Observe that $\pos_n(\bbC[G])$ contains all matrices of the form $B^*B$, is closed under positive linear combinations and 
 is closed in the topology of the character sup-norm.
 
 Let $\alpha\colon G \to H$ be a homomorphism of groups, then
 $\alpha$ induces a ring homomorphism $\bbC[G] \to \bbC[H]$ which will still be denoted by $\alpha$.
 The induced maps on matrices over $\bbC[G]$ will be denoted by $\alpha$ as well.
 
 \begin{lemma}\label{lem:sup-normFunctoriality}
   Let $\alpha\colon G \to H$ be a group homomorphism. For all $A \in M_n(\bbC[G])$ we have the inequality $\norm{\alpha(A)}_\infty \leq \norm{A}_\infty$
   and $\alpha\left(\pos_n(\bbC[G])\right) \subseteq \pos_n(\bbC[H])$.
 \end{lemma}
 \begin{proof}
 Let $\psi$ be a character of $H$ and let $\phi=\alpha^*(\psi)$.
 Consider the tracial Hilbert $H$-bimodule $(\calH_\psi,e)$ and define $V \subseteq \calH_\psi$ to be the closure of the subspace 
 spanned by $\{\: e\cdot\alpha(g) \:|\: g\in G \:\}$. The actions of $G$ by $\alpha(G)$ make $V$ into a tracial Hilbert $G$-bimodule
 of character $\phi$. By Proposition \ref{prop:Thoma} the bimodules $V$ and $\calH_\phi$ are isomorphic and 
 the operator norm of $l_{A}$ on $V^n$ agrees with the norm $\norm{l_A}_\phi$. 
 
 The right action by $H$ on $\calH_\psi^n$ commutes with $l_{\alpha(A)}$ and the space spanned by all the isometric translates $\{V^n\cdot h \:|\: h\in H\}$
 is dense in $\calH^n_\psi$. We deduce that
 \begin{equation*}
    \norm{l_{\alpha(A)}}_\psi = \norm{(l_{\alpha(A)})_{|V^n}}_\psi = \norm{l_A}_\phi,
 \end{equation*}
 and thus $\norm{\alpha(A)}_\infty \leq \norm{A}_\infty$.
 
 Suppose now that $A \in \pos_n(\bbC[G])$. Then $l_{\alpha(A)}$ is a positive self-adjoint operator on $V^n$ and as above 
 this holds on the translates $V^n\cdot h$ for all $h\in H$. As the translates span a dense subspace of $\calH_\psi^n$ the claim follows.
 \end{proof}

 \subsection{The spectral measure map}
 We will use the measure theoretic terminology from \cite[VIII]{Elstrodt2005}.
 In particular, a \emph{finite Radon measure} $\nu$ on $\bbR$ is a finite regular measure on the Borel sigma algebra of $\bbR$.
 We say that $\nu$ is \emph{compactly supported}, if there is a compact subset $C\subseteq \bbR$ such that $\nu(\bbR\setminus C) =0$.
 The set of finite compactly supported Radon measures on $\bbR$ will be denoted by $\calM^+_c(\bbR)$.
  
  \medskip 

 Given $A \in \pos_n(\bbC[G])$ and a character $\phi\in \Ch(G)$, we can associate to it the spectral measure $\mu_{A,\phi}$ of the self-adjoint left multiplication  operator $l_{A,\phi}$ on $\calH^n_\phi$ with respect to the $\calN_\phi$-trace.
 This means, $\mu_{A,\phi}$ is the unique finite Radon measure on $\bbR$ supported on the spectrum $\sigma(l_{A,\phi})$ such that for every 
 continuous function $f\in C(\bbR)$ the following formula holds:
 \begin{equation*}
      \Tr_{\calN_\phi}(f(l_{A,\phi})) = \int_{\bbR} f(t) d\mu_{A,\phi}(t).
 \end{equation*}
 
 \begin{theorem}\label{thm:spectralMeasureMap}
 The spectral measure map 
 \begin{equation*}
           \mu\colon \pos_n(\bbC[G])\times \Ch(G) \to \calM^+_c(\bbR)
 \end{equation*}
 which maps $(A,\phi)$ to $\mu_{A,\phi}$ is continuous with respect to
 \begin{itemize}
   \item the character sup-norm topology on $\pos_n(\bbC[G])$,
   \item the topology of pointwise convergence on $\Ch(G)$ and
   \item the topology of weak convergence on $\calM_c^+(\bbR)$.
 \end{itemize}
 Moreover, the map is affine in the second variable. 
 If $\alpha\colon G \to H$ is a group homomorphism, then
    $\mu_{\alpha(A),\psi} = \mu_{A, \alpha^*(\psi)}$
 for all $A \in \pos_n(\bbC[G])$ and $\psi \in \Ch(H)$.
 \end{theorem}
 \begin{proof}
 The following remark is in order: If the group $G$ is countable, then topology on the domain is metrizable.
 In this case the proof can be simplified by replacing nets by sequences.
 
 We show that for every convergent net $(A_j,\phi_j)_{j \in J} \to (A,\phi)$ in $\pos_n(\bbC[G])\times \Ch(G)$ over some directed set $J$, we have a weakly convergent net
  of measures
  \begin{equation*}
      \mu_{A_j,\phi_j} \stackrel{w}{\longrightarrow} \mu_{A,\phi}.
  \end{equation*}
  
  Since the net$(A_j)_{j\in J}$ converges to $A$ in the character sup-norm, we can assume that the norms $\norm{A_j}_\infty$ are bounded.
  We fix a positive number $C>0$ such that $\norm{A_j}\leq C$ for all $j \in J$.
  In particular, the spectral measures of $\mu_{A,\psi}$ and $\mu_{A_j,\psi}$ are supported on $[0,C]$ for any character $\psi$.
  
  Let $f\in C_b(\bbR)$ be a bounded continuous function on $\bbR$. We have to show that
  \begin{equation}\label{eq:weakConvergence}
      \lim_{j \in J} \int_\bbR f \: d\mu_{A_j,\phi_j} = \int_\bbR f \: d\mu_{A,\phi}.
  \end{equation}
  By the theorem of Stone-Weierstra{\ss}  we can approximate $f$ uniformly over the interval $[0,C]$ by polynomials.
  In particular, it is sufficient to prove \eqref{eq:weakConvergence} if $f$ is a polynomial. By linearity it is even sufficient to consider the
  case where $f(t) = t^k$ for some $k\in\bbN$.
  
  Note that the submultiplicativity of the character sup-norm (Lemma \ref{lem:charactersup-norm}) yields that the net $(A^k_j)_{j \in J}$ converges to $A^k$. 
  In addition, we have 
  \begin{equation}\label{eq:pushforwardFormula}
  f_*(\mu_{B,\eta}) = \mu_{B^k,\eta}
  \end{equation}
   (for all $B$ and $\eta$) and hence it suffices to verify the claim for $k = 1$.
  This means, we have to prove the identity
  \begin{equation*}
        \lim_{j \in J} \Tr_{\calN_{\phi_j}}(l_{A_j}) = \Tr_{\calN_\phi}(l_A).
  \end{equation*}
  With the triangle inequality we obtain
  \begin{equation*}
      \left| \Tr_{\calN_{\phi_j}}(l_{A_j})- \Tr_{\calN_\phi}(l_A)\right| \leq \left|\Tr_{\calN_{\phi_j}}(l_{A_j-A})\right| + \left| \Tr_{\calN_{\phi_j}}(l_{A}) - \Tr_{\calN_\phi}(l_{A})\right|.
  \end{equation*}
  The continuity of the trace shows that the first term satisfies
  \begin{equation*}
  \left|\Tr_{\calN_{\phi_j}}(l_{A_j-A})\right| \leq n \norm{l_{A_j- A}}_{\phi_j} \leq n \norm{A_j - A}_\infty
  \end{equation*}
  and hence it converges to zero.
  
  Now we consider the second term. By the definition of the trace it is sufficient to show that for every element $a\in \bbC[G]$
  the traces $\langle a\cdot e, e\rangle_{\phi_j}$ converge to $\langle a\cdot e, e\rangle_\phi$.
  Write $a$ as a finite linear combination of group elements, say $a = \sum_{i=1}^r \lambda_i g_i$.
  We conclude the proof of the continuity by observing
  \begin{equation*}
    \lim_{j \in J}\langle a\cdot e, e\rangle_{\phi_j} = \lim_{j \in J} \sum_{i=1}^r \lambda_i \phi_j(g_i) = 
    \sum_{i=1}^r \lambda_i \phi(g_i) = \langle a\cdot e, e\rangle_\phi.
  \end{equation*}
  
  Now we show that the spectral measure map is affine in the second variable.
  Assume that $\psi = \lambda \phi_1 + (1-\lambda) \phi_2$ 
  for $\psi,\phi_1,\phi_2 \in \Ch(G)$ and  $\lambda \in [0,1]$.
  We have to check that for every matrix $A\in \pos_n(\bbC[G])$ the spectral measures satisfy
 \begin{equation*}
 \mu_{A,\psi} = \lambda \mu_{A,\phi_1} + (1-\lambda)\mu_{A,\phi_2}.
 \end{equation*}
 By Riesz' representation theorem it is sufficient to verify that these measures integrate every continuous compactly supported function
 to the same value.
 As above, the measures are compactly supported, hence (using approximation) it is sufficient 
 to verify the claim for all monomial functions $f(t) = t^k$. 
 By \eqref{eq:pushforwardFormula} it is yet sufficient to treat the case $k=1$ which follows directly 
 from the assumption $\psi= \lambda \phi_1 + (1-\lambda) \phi_2$.
 
 Finally, the last assertion follows from exactly the same argument taking 
 into account Lemma \ref{lem:sup-normFunctoriality} which provides that $\alpha(A)$ is positive.
 \end{proof}
 
 \begin{corollary}\label{cor:PropertiesOfRank}
 Let $A \in M_{m,n}(\bbC[G])$.
 \begin{enumerate}
 \item\label{item:affine} The maps $\phi \mapsto \rnk_\phi(A)$ and $\phi \mapsto \nulli_\phi(A)$ are affine maps on $\Ch(G)$.
 \item\label{item:transformation} If $\alpha\colon G \to H$ is a group homomorphism and $\phi\in \Ch(H)$, then
 \begin{align*}
    \rnk_{\alpha*(\phi)}(A) &= \rnk_\phi(\alpha(A))\\
    \nulli_{\alpha*(\phi)}(A) &= \nulli_\phi(\alpha(A)).
 \end{align*}
 \item The rank $\rnk_\phi(A)$ is lower semi-continuous in $\phi$.
 \item The nullity $\nulli_\phi(A)$ is upper semi-continuous in $\phi$.
 \end{enumerate}
 \end{corollary}
 \begin{proof}
 By the dimension formula \eqref{eq:dim-formula} it is sufficient to consider the nullity in all assertions.
 Moreover, passing from $A$ to $A^*A$ we can assume that $A \in \pos_n(\bbC[G])$.
 Note that $\nulli_\phi(A) = \mu_{A,\phi}(\{0\})$ and the first two assertions follow directly from Theorem \ref{thm:spectralMeasureMap}.
 
 Let $(\phi_j)_{j \in J}$ be a net of characters converging to $\phi \in \Ch(G)$.
 By Theorem \ref{thm:spectralMeasureMap} we have weak
 convergence of spectral measures $\mu_{A,\phi_j} \stackrel{w}{\to} \mu_{A,\phi}$.
 The set $\{0\}$ is closed in $\bbR$, so the Portmanteau-Theorem (see \cite[VIII.~4.10]{Elstrodt2005} or for nets \cite[Thm.~1.3.4]{VaartWellner})
 implies
 \begin{equation*}
    \limsup_{j\in J} \: \nulli_{\phi_j}(A) = \limsup_{j \in J} \: \mu_{A,\phi_j}(\{0\}) \leq \mu_{A,\phi}(\{0\}) = \nulli_\phi(A). 
 \end{equation*} 
  \end{proof}
  
  \begin{example}
    In general the rank and nullity are \emph{not} continuous with respect to the character.
    This can be seen in the following simple example.
    Let $G = \langle g \rangle$ be an infinite cyclic group with generator $g$. 
    For every $z \in S^1 \subseteq \bbC$ there is a character $\hat{z}$ of $G$ given by
    $\hat{z}(g^k) = z^k$. The corresponding tracial Hilbert $G$-bimodule is a $1$-dimensional space
    on which $g$ acts by multiplication with $z$.
    
    Let $z_m \in S^1$ be a sequence of points on the unit circle which converges to~$1$.
    Clearly, the characters $\hat{z}_m$ converge pointwise to the trivial character $\triv = \hat{1}$.
    Assume $z_m \neq 1$ for all $m$ and consider the group ring element $a = 1 - g$.
    The ranks satisfy $\rnk_{\hat{z}_m}(a) = 1$ for all $m \in \bbN$, however $\rnk_{\triv}(a) = 0$.
  \end{example}

\begin{definition}\label{def:approximationProperty}
A convergent net of characters $(\phi_j)_{j \in J}$ with limit $\phi \in \Ch(G)$ is said to have
\emph{the approximation property} with respect a matrix $A \in M_{m,n}(\bbC[G])$ if
\begin{equation*}
 \lim_{j \in J} \rnk_{\phi_j}(A) = \rnk_\phi(A).
\end{equation*}
By \eqref{eq:dim-formula} this is equivalent to $\lim_{j \in J} \nulli_{\phi_j}(A) = \nulli_\phi(A)$.
\end{definition}

\subsection{The Fuglede-Kadison determinant}
In order to prove the approximation property for a sequence of characters one needs to control the spectrum close to zero. 
This can be achieved (as in L\"uck's approximation theorem) by bounding the Fuglede-Kadison determinant.
For an introduction to Fuglede-Kadison determinants we refer to \cite[3.2]{LuckBook}. Here we use a more measure theoretic terminology.

\begin{definition}
Let $\nu \in \calM^+_c(\bbR)$ be a finite, compactly supported Radon measure.
The \emph{Fuglede-Kadison determinant} of $\nu$ is
\begin{equation*}
    \det(\nu) = \exp \left( \int_{\bbR^\times} \ln|t| \: d\nu(t) \right)
\end{equation*}
with the convention that $\det(\nu) = 0$ if the integral diverges to $-\infty$.
\end{definition}

\begin{lemma}\label{lem:boundedDeterminant}
Let $(\nu_j)_{j \in J}$ be a net of measures in $\calM^+_c(\bbR)$ with supports in a compact interval $[0,C]$.
Assume that $(\nu_j)_{j\in J}$ converges weakly to
$\nu \in \calM^+_c(\bbR)$. If there is $B > 0$ such that
$\det(\nu_j) \geq B$ for all $j$, then $\det(\nu) \geq B$ and
\begin{equation*}
   \lim_{j \in J} \nu_j(\{0\}) = \nu(\{0\}).
\end{equation*}
\end{lemma}
\begin{proof}
Note that the support of $\nu$ is contained in the interval $[0,C]$ as well.
We fix a real number $V > 0$ such that $\nu_j(\bbR) \leq V$ for all large $j$ (and thus $\nu(\bbR) \leq V$). 
Let $f\colon \bbR\to \bbR$ be the function 
\begin{equation*}
     f(t) = \begin{cases}
              0 \quad& \text{ if } t \leq 0\\
              \ln(t) & \text{ if } 0 < t < C\\
              \ln(C) & \text{ if }t \geq C\\
     \end{cases}.
\end{equation*}
Then $\ln(\det(\nu_j)) = \int_\bbR f \: d\nu_j$ for all $j\in J$ (the same holds for $\nu$).
The function $f$ is upper semi-continuous and bounded from above, thus by the Portmanteau-Theorem \cite[Thm.~1.3.4]{VaartWellner} we obtain
\begin{equation*}
   B \leq \limsup_{j \in J} \det(\nu_j) \leq \det(\nu).
\end{equation*}

For all $1 > \varepsilon > 0$,  we obtain the estimate
\begin{equation*}
   \ln(B) \leq \ln(\det(\nu_j)) \leq \ln(\varepsilon) \nu_j\bigl((0,\varepsilon)\bigr) + \ln(C) V
\end{equation*}
and so $\nu_j\bigl((0,\varepsilon)\bigr) \leq \frac{y}{|\ln(\varepsilon)|}$ for some number $y > 0$.
Two applications of the Portmanteau-Theorem yield
\begin{align*}
\liminf_{j \in J}\nu_j(\{0\}) + \frac{y}{|\ln(\varepsilon)|} &\geq \liminf_{j \in J} \nu_j\bigl((-\varepsilon,\varepsilon)\bigr)
\geq \nu\bigl((-\varepsilon,\varepsilon)\bigr) \\
&= \nu\bigl( [0,\varepsilon) \bigr) \geq \limsup_{j \in J} \nu_j(\{0\}).
\end{align*}
The assertion follows by taking the limit $\varepsilon \to 0$.
\end{proof}

\subsection{Finite dimensional representations}\label{sec:finitedimensional}
Before we continue with the general discussion, we briefly review the introduced notions for the special case of characters
defined via finite dimensional unitary representations. These examples can be treated by means of elementary linear algebra, nevertheless, 
this case plays an important role later on.

Let $G$ be a group and let $\rho\colon G \to \U(V)$ be a  unitary representation on 
a finite dimensional Hilbert space $(V, \langle\cdot,\cdot\rangle)$.
The representation is completely reducible. 
This means, that we can decompose $V$ canonically into finitely many isotypic components
\begin{equation*}
   V = \bigoplus_{i=1}^s W_i
\end{equation*}
where each $W_i$ is isomorphic to $m_i$ copies of a finite dimensional irreducible unitary $G$-module $(\rho_i,V_i)$ and the representations 
$\rho_i$ are pairwise distinct.

We will show that the corresponding tracial Hilbert $G$-bimodule is the direct sum
\begin{equation*}
    \calH_\phi = \bigoplus_{i=1}^s \End_{\bbC}(V_i)
\end{equation*}
with the left (resp.\ right) action of $g\in G$ given by left (resp.\ right) multiplication with $(\rho_i(g))_i$.
More precisely,
the inner product of $a= (a_i)_{i}$ and $b = (b_i)_i$ is defined as $\langle a,b \rangle = \frac{1}{\dim(V)} \sum_{i=1}^s \Tr(b_i^* a_i)$ and the
tracial vector is $e = (\sqrt{m_i}\id_{V_i})_i$.
Indeed, this defines a Hilbert $G$-bimodule. The axioms \eqref{it:bimodule1} and \eqref{it:bimodule2} are evident, so it suffices to verify that the left translates of 
$e$ span a dense subspace.

Let $B$ be the $\bbC$-algebra spanned by $\{(\rho_i(g))_i \:|\: g\in G \}$. We claim that $B = \calH_\phi$.
Suppose this is not the case, then $B^{\perp} = \{\:x\in\calH_\phi\:|\: \langle b,x \rangle = 0 \text{ for all } b\in B\:\}$ is a non-trivial $B$-left submodule of $\calH_\phi$.
Let $W$ be an irreducible summand of $B^{\perp}$. The irreducible $B$-submodules of $\calH_\phi$ are all isomorphic to one of the modules $V_i$.
Since these are pairwise distinct the module $W$ lies in one isotypic component, i.e.\ one of the factors $\End_{\bbC}(V_j)$.
By a result of Burnside (see \cite[1.16]{Wehrfritz}) the projection map $B \to \End_{\bbC}(V_j)$ is surjective, which contradicts the assumption $W \subseteq B^{\perp}$.
We conclude that $\calH_\phi = B e$ as claimed and that $\calH_\phi$ is a tracial Hilbert $G$-bimodule with character
\begin{equation*}
  \langle g\cdot e, e\rangle = \frac{1}{\dim(V)}\sum_{i=1}^s m_i \Tr(\rho_i(g)) = \frac{1}{\dim(V)}\Tr(\rho(g)) = \phi(g).
\end{equation*}
Thus it is the bimodule associated to $\phi$.

Consider the von Neumann trace of a matrix $A = (a_{i,j}) \in \pos_n(\bbC[G])$ with respect to $\phi$. Write $l_A$ to denote the endomorphism defined by $A$ on $V^n$, then 
\begin{equation}\label{eq:vonNeumannTrace1}
   \Tr_{\calN_\phi}(A) = \frac{1}{\dim(V)}\sum_{i=1}^n \Tr(\rho(a_{i,i})) = \frac{1}{\dim(V)} \Tr(l_A).
\end{equation}
We can use this to determine the spectral measure of $A$ with respect to $\phi$ and its Fuglede-Kadison determinant.
Let $\lambda_1\leq \lambda_2 \leq \dots \leq \lambda_{n\dim(V)}$ denote the eigenvalues of $l_A$ (occurring with multiplicities).
\begin{lemma}\label{lem:finitedimFKdet}
 Let $A \in \pos_n(\bbC[G])$ and let $\rho\colon G \to \U(V)$ be a finite dimensional unitary representation with character $\phi$ as above.
 The spectral measure $\mu_{A,\phi}$ is 
 \begin{equation*}
    \mu_{A,\phi} = \frac{1}{\dim(V)} \sum_{i=1}^{n\dim(V)} \delta_{\lambda_i}
 \end{equation*}
 where $\delta_\lambda$ denotes the Dirac measure at $\lambda$.
 The Fuglede-Kadison determinant of $\mu_{A,\phi}$
 is $\det(\mu_{A,\phi}) = c^{1/\dim(V)}$, where $c$ is the
 product of the non-zero eigenvalues. In other words $c$ is the absolute value of the first non-zero coefficient of the characteristic polynomial of $l_A$.
\end{lemma}
\begin{proof}
  We use formula \eqref{eq:vonNeumannTrace1} to verify that the measure $\mu$ defined above has the universal property of the spectral measure.
  Indeed, it suffices to check this for monomial functions and
  \begin{equation*}
   \int_{\bbR} t^k \: d\mu(t) = \frac{1}{\dim(V)}\sum_{i} \lambda^k_i = \frac{1}{\dim(V)} \Tr( l_{A^k} ) = \Tr_{\calN_\phi}(A^k).
  \end{equation*}
 The statement about the Fuglede-Kadison determinant follows from a direct calculation of the integral using this measure.
\end{proof}

\section{Sofic and arithmetically hyperlinear characters}\label{sec:soficCharacters}
In this section we study characters which can be approximated by rather elementary characters.
For example, characters which can be approximated by characters of finite permutation representations.
This resembles the very general perspective on approximation taken by Farber in \cite{Farber1998}.
In order to include the sofic approximation theorem due to Elek-Szab\'o \cite{ElekSzabo2005},
it is necessary to allow approximation in the character space of a covering group.
We will first introduce the \emph{sofic characters} and relate them to sofic groups.
Then we generalize this to define \emph{arithmetically hyperlinear} characters.

\subsection{Sofic characters}
A character $\phi \in \Ch(G)$ is called a \emph{permutation character}
if there is 
an action of $G$ on a finite set $X$ such that $\phi(g) = |\fix_X(g)| \cdot|X|^{-1}$.
This is the character associated to the finite dimensional unitary representation of $G$ on the space of complex valued functions on $X$.
Note that the pullback of a permutation character is a permutation character.
The set of permutation characters of $G$ will be denoted by $\Perm(G)$.

\begin{definition}
Let $G$ be a group.
A character $\phi \in \Ch(G)$ is called \emph{sofic}, if there is an epimorphism $H \to G$
such that $\infl_G^H(\phi)$ lies in the closure of the set of permutation
characters of $H$.
\end{definition}
\begin{lemma}\label{lem:soficCharacters}
A character $\phi \in \Ch(G)$ is sofic if and only if for every epimorphism from
a free group onto $G$ the inflated character can be approximated by permutation characters.
Moreover, soficity is preserved under pullbacks along arbitrary homomorphisms.
\end{lemma}
\begin{proof}
Every group is the quotient of a free group, so the ``if'' statement is clear.
Conversely, assume that $\phi\in\Ch(G)$ is sofic and let $\alpha\colon H \to G$ be an epimorphism
such that $\alpha^*(\phi)$ lies in the closure $\overline{\Perm(H)}$ of the set of permutation characters.
Let $\pi\colon F \to G$ be any epimorphism from a free group $F$ onto $G$.
Since $F$ is free, there is a homomorphism $j\colon F \to H$ such that $\alpha\circ j = \pi$.
It follows from $j^*(\Perm(H)) \subseteq \Perm(F)$ and the continuity of $j^*$, that we have $\overline{\Perm(H)} \subseteq (j^*)^{-1}(\overline{\Perm(F)})$. We deduce that $\pi^*(\phi) = j^*(\alpha^*(\phi))$ lies in $\overline{\Perm(F)}$.

Let $\phi$ be a sofic character of $G$ and let $f\colon K \to G$ be a group homomorphism. We want to show that $f^*(\phi)$ is sofic.
Every homomorphism factors into a mono- and an epimorphism. We discuss these two cases separately.

Suppose $f$ is \emph{surjective} and let $j\colon F \to K$ be an epimorphism from a free group. Then $f\circ j \colon F \to G$ is an epimorphism and so  $j^*(f^*(\phi)) = (f\circ j)^*(\phi) \in \overline{\Perm(F)}$. We deduce that $f^*(\phi)$ is sofic as claimed.

Suppose $f$ is \emph{injective}, i.e.\ $K$ is a subgroup of $G$. Observe that the first part of the lemma shows that the claim holds if $G$ and $K$ are free groups.
The general case can be reduced to this case. Let $j\colon F \to G$ be an epimorphism from a free group. The inverse image $F_2 = j^{-1}(K)$
is again a free group (Nielsen-Schreier). The identity $\res^F_{F_2} \circ j^*(\phi) = j_{|F_2}^* \circ \res^G_K(\phi)$ shows that $\res^G_K(\phi)$ is sofic.
\end{proof}

The name sofic is justified by fact that this concept is closely related to the \emph{sofic groups} introduced by Gromov \cite{Gromov1999} under the name 
\emph{initially subamenable}. We briefly recall the definition given in \cite[Def.~1.2]{ElekSzabo2005} and we refer to \cite{ElekSzabo2006} for further properties.

\begin{definition}\label{def:soficgroup}
A group $G$ is \emph{sofic} if for every $c >0$ and every finite subset $W \subseteq G$,
there exists a finite set $X$ and a map $f \colon G \to \Aut(X)$ (i.g.\ no homomorphism) such that
\begin{enumerate}
\item for all $u,v \in W$ the set 
\begin{equation*}
L_{u,v} = \{\:x\in X\:|\: f(uv)(x) = (f(u)\circ f(v))(x) \:\}
\end{equation*}
has at least $(1-c)|X|$ elements.
\item $f(1) = \id_X$ and for all $u\in W\setminus \{1\}$ the fixed point set $\fix_X(f(u))$ has at most $c|X|$ elements. 
\end{enumerate}
\end{definition}
\begin{proposition}\label{prop:CharacterSofic}
A group $G$ is sofic if and only if its regular character $\delta^{(2)}_G$ is sofic.
\end{proposition}
\begin{proof}
Choose an epimorphism $\pi\colon F \to G$ where $F = F(Y)$ is a free group on a generating set $Y$.
Let $\phi = \pi^*(\delta^{(2)}_G)$ be the pullback of the regular character.

``$\implies$'': Assume that $G$ is sofic.  We will show that $\phi$ lies
in the closure of the permutation characters.

Let $\varepsilon > 0$, $r\in \bbN$  and $S \subseteq Y$ be a finite subset. Let $B(r,S)$ denote the finite set of words in $F$ of length at
 most $r$ using only the generators in $S$. We will show that there is a permutation character of $F$ which approximates $\phi$ on all
 elements in $B(r,S)$ up to $\varepsilon$. So, every neighbourhood of $\phi$ contains a permutation character and this implies that $\delta^{(2)}_G$ is sofic.
 
 Put $W = \pi(B(r,S))$. As $G$ is sofic, there is a finite set $X$ and a map $f\colon G \to \Aut(X)$ as in Definition \ref{def:soficgroup} 
 with $c = \frac{\varepsilon}{r+1}$. We obtain a homomorphism $\alpha\colon F \to \Aut(X)$ which sends $y \in Y$ to $f(\pi(y))$.
 We claim that $\alpha$ defines a suitable permutation representation of $F$.
 
 Indeed, by construction we have for all $w \in B(r,S)$
 \begin{equation*}
  \bigl\vert\{\:x \in X \:\vert\:  \alpha(w)(x) = f(\pi(w))(x) \:\}\bigr\vert \geq (1-\frac{r \varepsilon}{r+1})\vert X\vert
 \end{equation*}
 and further we have 
 \begin{equation*}
   \bigl\vert\fix_X(\alpha(w))\bigr\vert \leq \varepsilon \vert X \vert
 \end{equation*}
 for all $w \in B(r,S)$ with $\pi(w)\neq 1$.
 We conclude 
 $\vert\frac{\vert\fix_X(\alpha(w))\vert}{\vert X \vert} - \phi(w)\vert < \epsilon$ for all $w \in B(r,S)$.
 
 Conversely, assume that $\delta^{(2)}_G$ is sofic. This means, by Lemma \ref{lem:soficCharacters}, that $\phi$ can be approximated
 by permutation characters. Let $c>0$ and let $W \subset G$ be a finite subset. Choose a section $\sigma\colon G\to F$ of $\pi$ with $\sigma(1_G) = 1_F$.
 By our assumption there is a finite set $X$ and a permutation representation $\alpha\colon F \to \Aut(X)$ 
 such that 
 \begin{equation*} 
    \vert\frac{\vert\fix_X(\alpha(\sigma(w)))\vert}{\vert X \vert} - \delta^{(2)}(w)\vert < c
 \end{equation*}
 for all $w \in W$ and
 \begin{equation*}
   1 - \frac{\vert\fix_X(\alpha(\sigma(uv)^{-1}\sigma(u)\sigma(v)))\vert}{\vert X \vert}  < c
 \end{equation*}
 for all $u,v \in W$.
 Observe that $L_{u,v} = \fix_X(\alpha(\sigma(uv)^{-1}\sigma(u)\sigma(v)))$ for
 the map $f\colon G \to \Aut(X)$ defined by $f(g) = \alpha(\sigma(g))$.
 These inequalities imply that $f$ has the properties from Definition \ref{def:soficgroup}.

\end{proof}

Our interest in sofic characters stems from the following result which generalizes the approximation theorems of L\"uck \cite{Luck1994},
Bergeron-Gaboriau \cite{BergeronGaboriau2004} and Elek-Szab\'o \cite{ElekSzabo2005}.
\begin{theorem}\label{thm:soficCharactersApprox}
Let $G$ be a group.
Every convergent net of sofic characters in $\Ch(G)$ has the approximation property (Def.\ \ref{def:approximationProperty}) with respect
to every matrix over the group ring $\bbQ[G]$.
\end{theorem}
\begin{proof}
Since soficity and ranks are preserved by pullbacks (see Corollary \ref{cor:PropertiesOfRank} and Lemma \ref{lem:soficCharacters}) we can assume that $G$ is a free group.
Let $A$ be a matrix over the group ring $\bbQ[G]$.
We can assume, by clearing denominators, that $A$ lies in the integral group ring $\bbZ[G]$. Theorem \ref{thm:spectralMeasureMap} implies that a convergent net of characters
yields a weakly  convergent net of spectral measures. By Lemma \ref{lem:boundedDeterminant} 
it is sufficient to show that the Fuglede-Kadison determinant of $A^*A$ is bounded below on sofic characters.
The key observation is the following (often called the \emph{L\"uck Lemma}):
in every permutation representation the matrix $A^*A$ has non-zero integral Fuglede-Kadison determinant.
This is not difficult to verify. We postpone the
proof to the more general Lemma \ref{lem:lucklemma} below.
We use it to conclude that on permutation characters the Fuglede-Kadison determinant of $A^*A$ is bounded below by $1$. Since the Fuglede-Kadison determinant is upper semi-continuous
(Lemma \ref{lem:boundedDeterminant}) this bound is valid for all sofic characters.
\end{proof}

\subsection{Arithmetically hyperlinear characters}
We will now proceed to establish the approximation property for a more general class of characters. The procedure is the same as above, however,
we replace the set of permutation characters with the larger set of characters from representations with finite image.
To this end it is necessary to introduce some terminology.

\begin{definition}
 Let $\rho\colon G \to \GL(V)$ be a representation of $G$ on a finite dimensional $\bbC$-vector space.
 We say that $\rho$ is \emph{defined over an algebraic number field $F$}, if there is 
 a finite dimensional $F$-vector space $V_F$, a representation $\rho_F\colon G \to \GL(V_F)$ and an embedding $\iota\colon F \to \bbC$ 
 such that the representations $(\bbC \otimes_\iota V_F, \rho_F)$ and  $(V, \rho)$  are isomorphic.
 We will call the triple $(V_F,\rho_F,\iota)$ a \emph{model} of $\rho$ over $F$.
\end{definition}

\begin{definition}
Let $G$ be a group.
A  \emph{finite representation} of $G$ is a representation $\rho\colon G \to \GL(V)$ \emph{with finite image} on a finite dimensional complex vector space V.
\end{definition}

Note that every representation of a finite group $K$ is already defined over the cyclotomic field generated by the $|K|$-th roots of unity.
In particular, every finite representation is defined over some algebraic number field.
\begin{definition}
 Let $\rho$ be a finite representation of $G$.
 The \emph{arithmetic degree} of $\rho$ is the minimal degree of a number field over which $\rho$ is defined, i.e.\
\begin{equation*}
   \adeg(\rho) = \min\bigl\{\:[F:\bbQ] \:|\: \rho \text{ is defined over } F \:\bigr\}.
\end{equation*}
\end{definition}

% \begin{lemma}\label{lem:finiteGroups-arithmeticModels}
%  Let $G$ be a finite group and let $F$ be an algebraic $*$-field. Let $(V,\rho)$ be a finite dimensional representation of $G$ defined over $F$.
%  The there is a $G$-stable $\calO_F$-lattice $\Lambda \subseteq V$ and a hermitian form $\beta\colon \Lambda \times \Lambda \to \calO_F$ such that
%  $(F,\Lambda,\beta,\rho)$ is a finite arithmetic representation of $G$.
% \end{lemma}
% \begin{proof}
%  Since the group $G$ is finite, we can always achieve invariance by averaging over the group.
%  Hence we can proceed as follows: 1) choose a totally positive hermitian form $V \times V \to F$ and make into a $G$-invariant form $\beta$, 2) pick an full $\calO_F$-lattice $\Lambda$ in $V$ and make it $G$-stable and 3)
%  multiply $\beta$ with a suitable positive integer to achieve $\beta(\Lambda,\Lambda) \subseteq \calO_F$.
% \end{proof}

Let $(V,\rho)$ be a finite representation of $G$. 
We define the character $\phi$ attached to $\rho$ by
\begin{equation}\label{eq:attachedCharacter}
    \phi(g) = \frac{1}{\dim(V)} \: \Tr( \rho(g) ).
\end{equation}
Since $\rho(G)$ is finite, we can find a $G$-invariant hermitian form on $V$ by averaging.
Hence $\phi$ is the character of a finite dimensional unitary representation, however the concrete choice of a hermitian form on $V$ is irrelevant.
The characters of finite representations will be called the \emph{finite characters}. The number
$\adeg(\phi) := \adeg(\rho)$ is called the arithmetic degree of $\phi$. 
We denote the set of finite characters (of arithmetic degree at most $d$) of $G$ by
$\aCh(G)$ (respectively by $\aCh^d(G)$).

\begin{definition}
  Let $G$ be a group. A character $\phi \in \Ch(G)$ is called \emph{arithmetically hyperlinear} (of degree at most $d$), if there is an epimorphism $H \to G$
such that $\infl_G^H(\phi)$ lies in the closure of $\aCh(H)$ (respectively of $\aCh^d(H)$).
\end{definition}

As above it seems natural to call a group ``arithmetically hyperlinear'' if its regular character is. 
We leave it to the reader to find a non-character description of this class of groups.
The proof of Lemma \ref{lem:soficCharacters} generalizes to the case of arithmetically hyperlinear characters.

\begin{lemma}
A character $\phi \in \Ch(G)$ is arithmetically hyperlinear (of degree at most $d$) if and only if for every epimorphism from
a free group onto $G$ the inflated character can be approximated by finite characters (of arithmetic degree at most $d$).
This property is preserved under pullbacks.
\end{lemma}

The main theorem of this section is the following result, which partly generalizes Theorem 9.2 of Farber \cite{Farber1998}.
\begin{theorem}\label{thm:approxArith}
 Let $G$ be a group and $d\in\bbN$.
 Every convergent net of arithmetically hyperlinear characters in $\Ch(G)$ of degree at most $d$ has the approximation property with respect
 to every matrix over the group ring $\bbQ[G]$.
\end{theorem}
\begin{proof}
 As in Theorem \ref{thm:soficCharactersApprox} above we can assume that $G$ is a free group and it suffices to consider matrices $A$ over the integral group ring $\bbZ[G]$.
 By Lemma \ref{lem:boundedDeterminant} it is sufficient to find a lower bound for the Fuglede-Kadison determinant $\det(\mu_{A^*A,\phi})$ for all arithmetically hyperlinear characters $\phi$ of degree at most $d$
 which is independent of $\phi$.
 Since $G$ is free, every such character can be approximated by finite characters of arithmetic degree at most $d$. Hence, the upper semi-continuity of
 the Fuglede-Kadision determinant (Lemma \ref{lem:boundedDeterminant}) shows that it is sufficient to consider finite characters. The next lemma completes the proof.
\end{proof}
\begin{lemma}[L\"uck Lemma for finite characters]\label{lem:lucklemma}
 Let $A \in \pos_{n}(\bbZ[G])$ be a non-zero matrix over the integral group ring and
 let $\phi \in \Ch_f(G)$ be a finite character. 
 The Fuglede-Kadison determinant satisfies
 \begin{equation*}
   \det(\mu_{A,\phi}) \geq  \norm{A}^{-(d-1)n}_{\infty}
 \end{equation*}
  where $d = \adeg(\phi)$.
\end{lemma}
\begin{proof}
 Let $\rho\colon G \to \GL(V)$ be a finite representation with character $\phi$.
 Put $m = \dim_\bbC V$.
 Let $(V_F,\rho_F,\iota)$ be a model of $\rho$ over an algebraic number field $F$ of degree $d=[F:\bbQ]$.
 The ring of algebraic integers in $F$ is denoted by $\calO_F$.
 The image $\rho_F(G)$ is finite and hence (by averaging over the elements of $\rho_F(G)$) we can find a $G$-stable $\calO_F$-lattice\footnote{An $\calO_F$-lattice is a f.g.\ projective $\calO_F$-submodule of $V_F$ which spans $V_F$ over $F$.} 
 $\Lambda \subseteq V_F$.
 
 Let $l_A$ denote the left multiplication of $A$ on $\Lambda^n$. Given an embedding $v$ of $F$ into $\bbC$ we will denote the 
 left multiplication by $A$ on $\bbC \otimes_{v} V_F^n$ by $l^v_A$.
 
 We will use Lemma \ref{lem:finitedimFKdet} to compute the Fuglede-Kadison determinant.
 Let $c$ be the first non-zero coefficient of the characteristic polynomial of $l_A$. Since $\Lambda$ is a f.g.\ projective $\calO_F$-module we have $c \in \calO_F$.
 Observe that if $v$ denotes an embedding of $F$ into $\bbC$, then the
 corresponding action of $G$ on $V_v = \bbC \otimes_v \Lambda$ is an $m$-dimensional representation. Once again, since the action of $G$ factors over a finite quotient
 we can choose a $G$-invariant hermitian form on $V_v$; this means, the representation is an $m$-dimensional unitary representation.

 By Lemma \ref{lem:finitedimFKdet} the Fuglede-Kadison determinant of $l^v_A$ is
 $|v(c)|^{1/m}$. Since $c$ is a non-zero algebraic integer, the norm $N_{F/\bbQ}(c) = \prod_v v(c)$ is a non-zero integer.
 We conclude that
 \begin{equation*}
    \prod_{v} |v(c)| \geq 1
 \end{equation*}
 where the product runs over all $d$ embeddings\footnote{Here we really mean embeddings and not places in the sense of number theory. This means, an embedding and its complex conjugate
 occur both in the product if they are distinct.} of $F$ into $\bbC$.
 Moreover, $|v(c)|$ is a product of eigenvalues of $l^v_A$. The eigenvalues are bounded above by $\norm{A}_\infty$ and there are at most $nm$ non-zero eigenvalues, hence we obtain the estimate
 \begin{equation*}
     \det(\mu_{A,\phi})^m =  |\iota(c)| \geq \norm{A}_\infty^{-(d-1)nm}.
 \end{equation*}
\end{proof}

\section{Induction of characters from finite groups}\label{sec:Induction}

Let $G$ be a group and let $H$ be a finite group. In this section we will explain how to induce characters from $H$ to $G$.
In particular, we want this induction to agree with the usual induction of characters on finite groups. We begin with a short account of generalized induction 
and then we discuss the case of finite groups to motivate the definitions.

\subsection{Generalized induction using bimodules}
Fix a commutative ring $R$. We briefly remind the reader of generalized induction and restriction using bimodules.
Let $G$ and $H$ be groups and let $M$ be a $R[G]$-$R[H]$-bimodule.
The bimodule $M$ can be used to define two functors. First,
\emph{induction} of right $R[H]$- to right $R[G]$-modules by taking $R[H]$-homomorphisms $M$:
\begin{equation*}
    \Ind_M(V) := \Hom_{R[H]}(M, V)
\end{equation*}
for every right $R[H]$-module $V$.
Where the right $R[G]$-module structure arises from the $G$-action $(f\cdot g)(x) = f(gx)$.
Second, \emph{restriction} of right $R[G]$-modules to right $R[H]$-modules by tensoring:
\begin{equation*}
   \Res_M(W) := W \otimes_{R[G]} M
\end{equation*}
for every $R[G]$-module $W$. 

\begin{example}\label{ex:generalizedRestriction}
 Suppose $H$ is a subgroup of $G$. If $M = R[G]$, then $\Ind_M$ is the usual (co-)induction from $H$ to $G$ and $\Res_M$ is the usual restriction from $H$ to $G$.
 
 Here is another example: take a subgroup $G_0 \leq G$ such that $H \subseteq N_G(G_0)$ and consider $X = G/G_0$, which admits a left action by $G$ and a right action by $H$.
 In this case $M = R[X]$ is a $R[G]$-$R[H]$-bimodule and $\Res_M(W) \cong W_{G_0}$ is the module of $G_0$-coinvariants of $W$ with the action restricted to $H$.
\end{example}

The tensor-hom adjunction yields the following version of Frobenius reciprocity.
\begin{lemma}\label{lem:FrobeniusReciprocity}
   The functors $\Ind_M$ and $\Res_M$ are adjoint, i.e.\ there is a natural isomorphism of $R$-modules
   \begin{equation*}
      \Hom_{R[H]}(\Res_M(W), V ) \cong \Hom_{R[G]}(W, \Ind_M(V)).
   \end{equation*}
   for every $R[H]$-module $V$ and $R[G]$-module $W$.
\end{lemma}
\begin{proof}
 See  \cite[II.\textsection4]{BourbakiAlgebra1}.
\end{proof}

\subsection{Finite groups}
In this section $G$ and $H$ denote finite groups.
\begin{remark}
 The irreducible characters of $G$ in the sense of the representation theory of finite groups will be denoted by $\Irr(G)$.
It is important to note the following: if $\chi \in \Irr(G)$, then $\chi(1)$ is the dimension of the underlying representation, which is in general larger than one.
In this case $\chi$ is not a character in the sense of Definition~\ref{def:characters}. 
However, the normalized class function $\widetilde{\chi}(g) = \frac{\chi(g)}{\chi(1)}$ defines a character
in the sense of Definition~\ref{def:characters}. Indeed, $\Ch(G)$ is the convex hull of the characters $\{\widetilde{\chi}\: | \: \chi \in \Irr(G)\}$.
For emphasis the characters of $G$ in the sense of representation theory will sometimes be called \emph{ordinary characters}.
\end{remark}

\begin{lemma}\label{lem:ArithmeticDegreeFiniteGroup}
 The character space of a finite group $G$ is the closure of the finite characters of arithmetic degree at most $d=\varphi(|G|)$, i.e.\ $\overline{\aCh^d(G)} = \Ch(G)$ .
\end{lemma}
\begin{proof}
 Let $F$ be the number field generated by all the $|G|$-th roots of unity.
 Every finite dimensional representation
 of $G$ is defined over the field $F$, thus the characters $\widetilde{\chi}$ with $\chi \in \Irr(G)$ (and all rational convex combinations) are finite
 with $\adeg(\chi) \leq [F:\bbQ] = \varphi(|G|)$ where $\varphi$ denotes Euler's totient function.
 Finally, every convex combination of the characters $\widetilde{\chi}$ can be approximated by rational ones and the assertion follows.
\end{proof}

Let $M$ be a finite dimensional $\bbC[G]$-$\bbC[H]$-bimodule. We can interpret it as a complex 
representation (from the left!) of the direct product $G\times H$ (by making the right action of $H$ into a left action via inverses).
As such it has an ordinary character $\psi\colon G\times H \to \bbC$; more precisely, $\psi(g,h) = \Tr( m \mapsto gmh^{-1} | M)$.
This is also the character of the dual bimodule $M^*$ with the action written from the right.
\begin{lemma}\label{lem:inducedCharacterFormula}
  Let $M$ be a $\bbC[G]$-$\bbC[H]$-bimodule with ordinary character $\psi$ and let $V$ be a finite dimensional right $\bbC[H]$-module with ordinary character $\chi$.
  The ordinary character of $\Ind_M(V)$ is
  \begin{equation}\label{eq:inducedCharacterFormula}
    \Ind_{\psi}(\chi)(g) := \frac{1}{|H|} \sum_{h \in H} \psi(g,h)\chi(h).
  \end{equation}
\end{lemma}
\begin{proof}
  The $G$-module $\Hom_{\bbC[H]}(M, V)$ is the $G$-module $\Hom_{\bbC}(M, V)^H$ of $H$ invariants in $\Hom_{\bbC}(M, V) \cong M^* \otimes_\bbC V$
  (where $H$ acts like $(\alpha \cdot h)(m) = \alpha(mh^{-1})h$). The space $\Hom_{\bbC}(M, V)^H$ is a direct summand of $\Hom_{\bbC}(M, V)$ as $G$-module and the
  $G$-equivariant projection 
  $P \colon \Hom_{\bbC}(M, V) \to \Hom_{\bbC[H]}(M, V)$ is given by averaging over the $H$-action.
  Hence, we can compute the trace of $g\in G$ on $\Hom_{\bbC[H]}(M, V)$ as 
  \begin{equation*}
     \Tr(g | (M^* \otimes_\bbC V)^H) = \Tr( P \circ g | M^* \otimes_\bbC V ) = \frac{1}{|H|} \sum_{h\in H} \psi(g,h) \chi(h). \qedhere
  \end{equation*}
\end{proof}

\begin{example}
  Let us take a look at the example of ordinary induction of a character $\chi$ from a subgroup $H \leq G$ to the group $G$.
Let $M = \bbC[G]$ with the left $G$- and right $H$-action. The character $c_G$ of this bimodule is 
\begin{equation}\label{eq:c-function}
   c_G(g,h) = \bigl|\bigl\{\:f \in G\:|\: gfh^{-1} = f\:\bigr\}\bigr| = \begin{cases}
                                                                          |C_G(h)| \quad & \text{ if $g$ and $h$ are $G$-conjugate,}\\
                                                                             0 & \text{ otherwise. }
                                                                        \end{cases}
\end{equation}
With a small calculation we obtain the well-known formula for the induction of characters
\begin{equation*}
   \Ind_H^G(\chi)(g) = \frac{1}{|H|}\sum_{h\in H}  c_G(g,h) \: \chi(h)  = \frac{1}{|H|}\sum_{t\in G} \dot{\chi}(tgt^{-1}) .
\end{equation*}
where $\dot{\chi}$ denotes the function on $G$ which agrees with $\chi$ on $H$ and vanishes outside of $H$.
\end{example}

\subsection{Basic properties of induced characters}
We return to the general setting where $G$ is a group and $H$ is a finite group.
Observe that formula \eqref{eq:inducedCharacterFormula} does not use that the group $G$ is finite. We will simply use this formula to induce characters 
from $H$ to $G$ via characters of $G\times H$.
However, this time we will have to normalize the characters appropriately.

\begin{deflem}\label{deflem:induction}
Let $G$ and $H$ be groups and assume that $H$ is finite.
  For characters $\phi \in \Ch(H)$ and $\psi \in \Ch(G\times H)$ such that $\sum_{h \in H} \psi(1,h) \: \phi(h) \neq 0$ the formula 
  \begin{equation}\label{eq:induced}
    \ind_\psi(\phi)(g) = \frac{\sum_{h \in H} \psi(g,h) \: \phi(h)}{\sum_{h \in H} \psi(1,h) \: \phi(h)}
  \end{equation}
   defines a character on $G$, which is called the \emph{character induced by $\phi$ via $\psi$}.
   If $\sum_{h \in H} \psi(1,h) \: \phi(h) \neq 0$ we say that $\phi$ \emph{can be induced via} $\psi$.
\end{deflem}
\begin{proof}
 For simplicity we write $B = \sum_{h \in H} \psi(1,h) \: \phi(h)$ and we note that $B > 0$.
 Since $\psi$ is a class function on $G\times H$, it is clear that $\ind_\psi(\phi)$ is a class function on $G$.
 By construction $ \ind_\psi(\phi)(1) = 1$, hence it remains to show that $\ind_\psi(\phi)$ is of positive type.
 Note that the function $(g,h) \mapsto \psi(g,h) \phi(h)$ defines a character on $G\times H$.
 By the GNS construction this is the matrix coefficient $\psi(g,h)\phi(h) = \langle (g,h)v,v \rangle$
 for some representation $G\times H \to \U(V)$ 
 on a Hilbert space $V$ (see \cite[(3.20)]{Folland1995}). Define $w = \frac{1}{\sqrt{|H| B}} \sum_{h \in H} (1,h)\cdot v$ and observe that
  \begin{equation*}
    \ind_\psi(\phi)(g) = B^{-1} \sum_{h \in H} \psi(g,h) \: \phi(h) = \langle (g,1) w, w \rangle
  \end{equation*}
  is a matrix coefficient as well.
\end{proof}

  Let $G$ be a group and let $H \leq G$ be a finite subgroup. In this situation there is a character on $G\times H$
  which corresponds to the usual notion of induction in the setting of finite groups. In fact, we will simply use 
  the function $c_G$ in \eqref{eq:c-function} and normalize it.
  \begin{lemma}
   Let $G$ be a group and let $H\leq G$ be a finite subgroup.
   The function $i_G\colon G\times H \to \bbC$ defined as
   \begin{equation}\label{eq:i-function}
        i_G(g,h) = \begin{cases}
                       [G:C_G(h)]^{-1} \quad & \text{ if $g$ and $h$ are conjugate}\\
                       0 & \text{ otherwise.}
                   \end{cases}
   \end{equation}
   is a character on $G \times H$.
  \end{lemma}
  \begin{remark}
    Here we make the convention that $[G:C_G(h)]^{-1}$ is zero if the index is infinite.
  \end{remark}
  \begin{proof}
   Clearly, $i_G$ is a class function and $i_G(1,1) = 1$.
   We proceed to show that it is of positive type.
 
 The set  $H_0 =\{h\in H\:|\:[G:C_G(h)] < \infty\}$ is a normal subgroup of $H$. Indeed, for all $h_1, h_2 \in H_0$ we have
 $C_G(h_1h_2) \supseteq C_G(h_1) \cap C_G(h_2)$ and consequently $C_G(h_1h_2)$ has finite index in $G$. 
 The function $i_G(g,h) = 0 $ for all $h \not \in H_0$. Extending a function by $0$ to a supergroup preserves positivity,
 so we may assume that $H_0 = H$; i.e.\
 the centralizer $C_G(h)$ of every element of $h \in H$ has finite index in $G$.
 
 Take a finite index normal subgroup $N \normal G$ which is contained in the centralizer of every $h\in H$.
 Put $d_g(x,h) = 1$ if $x = ghg^{-1}$ and $d_g(x,h) = 0$ otherwise.
 The set of elements in $G$ which conjugate $h \in H$ to $x \in G$ is the union of $[C_G(h) : N]$ cosets of $N$.
 This shows that $i_G(x,h) = \frac{1}{[G:N]}\sum_{g \in G/N} d_g(x,h)$.
 The space of functions of positive type is closed under positive linear combinations, thus it is sufficient to check that the functions $d_g$ are of positive type.
 
 Consider the Hilbert space $L^2(G)$ with the action of $G\times H$ defined by $(x,h)\cdot g = xgh^{-1}$.
 Since $d_g(x,h) = \langle (x,h) \cdot g, g\rangle$, this is a function of positive type.
\end{proof}

\begin{definition}
 Let $G$ be a group and let $H$ be a finite subgroup.
 For $\phi \in \Ch(H)$ the character induced by $\phi$ to $G$ via $i_G$ will be denoted
 by $\ind_H^G(\phi)$.
\end{definition}
Observe that $\sum_{h\in H} i_G(1,h) \phi(h) = 1$ and so every character can be induced via $i_G$ and the normalization in \eqref{eq:induced} is not necessary in this case.
\begin{example}
 Let $G$ be a group and let $H = \{1\}$ be the trivial subgroup.
 The group $H$ has only the trivial character $\triv$. The character induced to $G$ from the trivial character on $\{1\}$ is the regular character $\delta^{(2)}_G$ of $G$.
\end{example}
\begin{lemma}\label{lem:convergenceInducedCharacters}
 Let $H\leq G$ be a finite subgroup. Suppose that $N_1 \supset N_2 \supset N_3 \supset \dots$ is a decreasing sequence of normal subgroups of $G$ with
 $\bigcap_{k} N_k = \{1\}$. Assume $N_k \cap H = \{1\}$ for all $k$ and write $G_k = G/N_k$.
 Then
 \begin{equation*}
  \lim_{k\to \infty} \infl^G_{G_k} \ind_{H}^{G_k}(\phi)  = \ind_H^G(\phi)
 \end{equation*}
 for every character $\phi \in \Ch(H)$.
\end{lemma}
\begin{proof}
The identity
 $\infl^G_{G_k} \ind_{H}^{G_k}(\phi)(g) = \sum_{h\in H} \phi(h) i_{G_k}(gN_k,hN_k)$
shows that it is sufficient to verify
\begin{equation}\label{eq:limit-of-i}
   \lim_{k\to\infty} i_{G_k}(gN_k,hN_k) = i_G(g,h)
\end{equation}
for all $g\in G$ and $h\in H$.

\medskip
 
\noindent\emph{Claim:} $\lim_{k\to\infty} [G_k : C_{G_k}(xN_k)] = [G:C_G(x)]$ for all $x\in G$.\\
Note that $[G_k : C_{G_k}(xN_k)] \leq [G:C_G(x)]$ for all $k\in \bbN$.
Let $Y$ be a set of left coset representatives of $C_G(x)$ in $G$. For every finite subset $S \subseteq Y$ 
the elements $\{sxs^{-1} \:|\: s\in S \}$ are distinct modulo $N_k$ for all sufficiently large $k$.
We have $[G_k:C_{G_k}(xN_k)] \geq  |S|$ for all large $k$. Hence, if $Y$ is finite the claim follows with $|S| = |Y| = [G:C_G(x)]$.
On the other hand, if $Y$ is infinite, then the argument shows that $[G_k:C_{G_k}(xN_k)]$ tends to infinity with $k$.

\medskip

\noindent The verification of \eqref{eq:limit-of-i} is done by case distinction.\\
\emph{Case 1}: $g$ and $h$ are conjugate in $G$. In this case $gN_k$ and $hN_k$ are conjugate in $G_k$ (for all $k$). The claim above immediately implies \eqref{eq:limit-of-i}.\\
\emph{Case 2}: $g$ and $h$ are not conjugate in $G$. If $[G:C_G(h)]$ is infinite, then once again the claim again yields $\lim_{k\to\infty}[G_k:C_{G_k}(hN_k)]^{-1} = 0 $ and \eqref{eq:limit-of-i}
follows. Finally, assume that $[G:C_G(h)] < \infty$. Then $h$ has only a finite set $S$ of $G$-conjugates. For all sufficiently large $k$ the $\{sN_k\:|\:s\in S\}$ are distinct from $gN_k$
and so $hN_k$ is not conjugate to $gN_k$ for all large $k \in \bbN$.
\end{proof}

\subsection{Induction via arithmetically hyperlinear characters}
Lemma \ref{lem:convergenceInducedCharacters} yields the pointwise convergence of certain
sequences of induced characters. In order to prove the approximation property for this sequence it is sufficient (in view of Theorem \ref{thm:approxArith}) 
to find conditions under which the characters are arithmetically hyperlinear of bounded degree.
Before we give the result we need the following variation of the concept of arithmetically hyperlinear (resp.\ sofic)
characters.
Let $G$ and $H$ be groups. A character $\psi \in \Ch(G\times H)$ is arithmetically hyperlinear of degree at most $d$ (resp.\ sofic) relative to $H$, if there is an epimorphism
$\pi\colon \widetilde{G} \to G$ such that $\infl_{G\times H}^{\widetilde{G} \times H} (\psi)$ lies in the closure of $\aCh^d(\widetilde{G}\times H)$ (resp.\ $\Perm(\widetilde{G}\times H)$).
\begin{theorem}\label{thm:inductionViaArithmeticallyHyperlinear}
  Let $G$ be a group and let $H$ be a finite group.
  Let $\phi \in \Ch(H)$ and $\psi \in \Ch(G\times H)$ be characters such that $\phi$ can be induced via $\psi$. If $\psi$ is arithmetically hyperlinear of degree at most $d$ relative to $H$,
  then $\ind_\psi(\phi)$ is arithmetically hyperlinear of degree at most $d\varphi(|H|)$.
\end{theorem}
\begin{proof}
 Put $e = \varphi(|H|)$.
 Let $\pi \colon \widetilde{G} \to G$ be an epimorphism such that $\vartheta = \infl_{G\times H}^{\widetilde{G} \times H} (\psi)$
 lies in $\overline{\aCh^d(\widetilde{G}\times H)}$.
 Since $\infl_{G}^{\widetilde{G}}(\ind_\psi(\phi)) = \ind_\vartheta(\phi)$, it is sufficient to verify that $\ind_\vartheta(\phi)$
 can be approximated by finite characters of arithmetic degree at most $de$.
 
 Choose a net $(\vartheta_j)_{j \in J} \in \aCh^d(\widetilde{G}\times H)$ of finite characters which converges pointwise to 
 $\vartheta$.
 By Lemma \ref{lem:ArithmeticDegreeFiniteGroup} there is a sequence of finite characters $(\phi_n)_{n\in \bbN} \in \Ch^e_f(H)$ which converges to 
 $\phi$. 
 Observe that the net of sums  $S_{j,n} = \sum_{h \in H} \vartheta_j(1,h) \: \phi_n(h)$ over $J\times \bbN$ converges to the non-zero sum $S = \sum_{h \in H} \vartheta(1,h) \: \phi(h)$.
 We conclude that $S_{j,n} \neq 0$ for sufficiently large $j$ and $n$; so $\phi_n$ can be induced via $\vartheta_j$.
 
 We verify that the net of characters $\ind_{\vartheta_j}(\phi_n)$ converges to $\ind_\vartheta(\phi)$.
 In fact, for all $g \in \widetilde{G}$ and $h \in H$ we obtain
 \begin{align*}
     \lim_{(j,n) \in J\times\bbN} \ind_{\vartheta_j}(\phi_n)(g,h) &=  \lim_{(j,n) \in J\times\bbN} \frac{1}{S_{j,n}} \sum_{h \in H} \vartheta_j(g,h) \: \phi_n(h) \\
     &= \frac{1}{S} \sum_{h \in H} \vartheta(g,h)\: \phi(h) = \ind_{\vartheta}(\phi)(g,h).
 \end{align*}

 Eventually we prove that for fixed $j\in J$ and $n\in \bbN$ the character $\ind_{\vartheta_n}(\phi_n)$ is finite of arithmetic degree at most $de$.
 Let $\rho_j\colon \widetilde{G}\times H \to \GL(V_j)$ and $\eta_n\colon H \to \GL(W_n)$ denote representations attached to $\vartheta_j$ and $\phi_n$.
 Let $F,E$ be algebraic number fields of degree $d = [F:\bbQ]$ and $e= [E:\bbQ]$ such that
 $\rho_j$ is defined over $F$ and $\eta_n$ is defined over $E$. The induced character $\ind_{\vartheta_j}(\phi_n)$
 is the character of the $\widetilde{G}$-representation on $\Ind_{V_j}(W_n)$ (see Lemma \ref{lem:inducedCharacterFormula}).
 This representation is finite and can be defined over the field compositum $L = FE$.
 This finishes the proof, since $[L : \bbQ] \leq de$.
\end{proof}

\subsection{Relatively sofic groups}
Our construction of induced characters leads to the following definition of relatively sofic groups.
We discuss some basic properties.
\begin{definition}
   Let $G$ be a group and $H \leq G$ be a subgroup.
   We say that $G$ is \emph{sofic relative to $H$} if the character $i_G\in\Ch(G\times H)$ is sofic relative to $H$.
\end{definition}
\begin{question}
  Let $G$ be a sofic group and let $H\leq G$ be a finite subgroup. Is $G$ sofic relative to $H$?
\end{question}
In view of our applications this question arises quite naturally.
The following result will provide a partial answer. It illustrates that there is a large class of such groups which clearly contains all 
residually finite groups.

\begin{proposition}\label{prop:criterionRelativeSoficity}
  Let $G$ be a sofic group and let $H \leq G$ be a finite subgroup.
  If there is a finite index subgroup $N \leq_{f.i.} G$ such that $N \cap H = \{1\}$, then $G$ is sofic relative to $H$.
\end{proposition}
\begin{corollary}
  Let $G$ be a sofic group and let $H$ be a finite group acting on $G$ by automorphisms. The semidirect product $G \rtimes H$ is sofic relative to $H$.
\end{corollary}
\begin{proof}[Proof of the corollary]
 This follows from the proposition and the fact that the semidirect product $G \rtimes H$ is sofic by \cite[Thm.~1]{ElekSzabo2006}. 
\end{proof}

Before we give the proof of the proposition, 
we discuss a lemma which allows us to reduce to the case where the centralizer $C_G(h)$ of every element $h \in H$ has finite index in $G$.
\begin{lemma}\label{lem:reduceToCentral}
 Let $G$ be a group and let $H$ be a finite subgroup. 
 Let $H_0 =\{\:h \in H\:|\: [G:C_G(h)] < \infty\:\}$  and 
 let $K \leq_{f.i.} G$ be a finite index subgroup which contains $H_0$.
 If $K$ is sofic relative to $H_0$, then $G$ is sofic relative to $H$.
\end{lemma}
\begin{proof}
 Note that $H_0 \normal H$ is a normal subgroup of $H$.
 Intersecting the finitely many $H$-conjugates of $K$, we may assume that $K$ is normalized by $H$.
 Pick a surjection $\pi\colon F \to G$ from a free group $F$.
 The inverse image $\pi^{-1}(K)$ will be denoted by $\widetilde{F}$; it is again a free group. 
 Let $\gamma_1,\dots,\gamma_s$ be coset representatives for $F/\widetilde{F}$ and let $b_1,\dots,b_m$ be coset representatives of $H/H_0$.
 
 Let $\varepsilon > 0$ and $f_1,\dots,f_r \in F$ be given.
 We have to find a finite $F$-$H$-biset, whose character is $\varepsilon$-close to the character $\infl^{F\times H}_{G \times H}(i_G)$ at 
 the elements $\{f_1,\dots,f_r\} \times H$.
 By assumption there is a finite $\widetilde{F}$-$H_0$-biset $X$ such that 
 \begin{equation*}
    \left| \frac{|\fix_X(\gamma^{-1}_i f_j \gamma_i,h)|}{|X|} - i_K(\pi(\gamma^{-1}_i f_j \gamma_i),h)\right| < \varepsilon.
 \end{equation*}
 for all $h\in H_0$, all $i\in\{1,\dots,s\}$ and all $j\in\{1,\dots,r\}$ with $\gamma_i^{-1}f_j \gamma_i \in \widetilde{F}$.
  Consider the set $Y = (F\times X \times H)/\sim$ where the equivalence relation $(g_1,x_1,h_1) \sim (g_2,x_2,h_2)$  holds if there are $g\in \widetilde{F}$ and $t\in H_0$ such that
  $g_1g^{-1} = g_2$, $gx_1t^{-1}=x_2$ and $th_1 = h_2$. Then $Y$ is carries a left action by $F$ and a right action by $H$. Every element of $Y$ has a unique representative of the form
  $(\gamma_i, x, b_j)$ and thus $|Y| = |F/\widetilde{F}|\cdot|H/H_0|\cdot|X|$.
  We will verify that the character of $Y$ is $\varepsilon$-close to $\infl(i_G)$ at all elements $(f_j,h)$.
 
 Fix $f \in \{f_1,\dots, f_r\}$ and $h \in H$. The point $y = (\gamma_i,x,b_j)$ is an $(f,h)$-fixed point if and only if
 1) $\gamma_i^{-1}f\gamma_i \in \widetilde{F}$, 2) $h\in H_0$ and 3) $(\gamma_i^{-1}f\gamma_i) \cdot x \cdot (b_jh^{-1}b_j^{-1}) = x$.
 If $h$ is not in $H_0$, then $i_G(\pi(f),h) = 0$ and there is nothing to show.
 Now assume that $h \in H_0$, then we obtain the formula
 \begin{align*}
     |\fix_Y(f,h)| =& \sum_{\substack{i=1 \\ \gamma_i^{-1}f\gamma_i \in \widetilde{F}}}^s \sum_{j=1}^m |\fix_X(\gamma_i^{-1}f\gamma_i,b_jhb_j^{-1})|\\
                   =& \sum_{\substack{i=1 \\ \gamma_i^{-1}f\gamma_i \in \widetilde{F}}}^s \sum_{j=1}^m |X| \: i_K(\pi(\gamma_i^{-1}f\gamma_i),b_jhb_j^{-1}) + O(|Y|\varepsilon)
 \end{align*}
 with an error term bounded by $|Y| \varepsilon$.
 The term $i_K(\pi(\gamma_i^{-1}f\gamma_i),b_jhb_j^{-1})$ vanishes if $\pi(\gamma_i^{-1}f\gamma_i)$ is not conjugate to $h_j = b_jhb_j^{-1}$ in $K$.
 Therefore we can restrict the summation to the set of indices
 $I_f(h_j) =\{\:i \in\{1,\dots,s\}\:|\: \pi(\gamma_i^{-1} f \gamma_i) \sim_K h_j \:\}$.
 Note that $F/\widetilde{F} \cong G/K$.  This shows that the set $I_f(h_j)$ is either empty  (if $\pi(f)$ is not conjugate to $h_j$ in $G$) or otherwise has exactly $[C_G(h_j) : C_K(h_j)]$ elements.
 Conjugation with $b_j$ defines an automorphism of $G$ which preserves $K$. We deduce that $[C_G(h_j) : C_K(h_j)]$, $[K:C_K(h_j)]$ and  
  $|I_f(h_j)|$ are independent of $j$.
 We obtain
 \begin{equation*}
     \frac{|\fix_Y(f,h)|}{|Y|} = \frac{1}{m s}\sum_{j=1}^m\sum_{i \in I_f(h_j)} \: [K: C_K(h)]^{-1} + O(\varepsilon) = i_G(\pi(f),h) + O(\varepsilon).
 \end{equation*}
 with an error term which is bounded by $\varepsilon$.
\end{proof}

\begin{proof}[Proof of Proposition \ref{prop:criterionRelativeSoficity}]
 Take a finite index normal subgroup $N$ such that $N \cap H = \{1\}$.
 By Lemma \ref{lem:reduceToCentral} we may assume that every element $h\in H$ has a finite index centralizer in $G$; we may further assume that 
 every element of $N$ centralizes $H$ by possibly choosing $N$ smaller. Note that the group $K = NH$ is isomorphic to the direct product $N \times H$. Another application of Lemma \ref{lem:reduceToCentral}
 shows that it suffices to verify that $K$ is sofic relative to $H$.
 
 It remains to show that a direct product $K = N \times H$ of a sofic group $N$ and a finite group $H$ is sofic relative to $H$.
 Let $\pi\colon F \to N$ be an epimorphism from a free group $F$ and let $\phi = \infl^F_N(\delta^{(2)}_N)$ be the inflation of the regular character.
 By Proposition \ref{prop:CharacterSofic} there is a net of permutation characters $\phi_j \in \Perm(F)$ which converges to~$\phi$.
 The character $i_H\in \Ch(H\times H)$ (defined as in \eqref{eq:i-function} with $i_H(h_1,h_2) = [H:C_H(h_2)]^{-1}$ if $h_1$ and $h_2$ are conjugate) 
 is the permutation character of the action of $H\times H$ on $H$ via $(h_1,h_2)\cdot x = h_1 x h^{-1}_2$.
 Observe that $[K:C_K(h)] = [H:C_H(h)]$ for all $h\in H$ and consequently
 \begin{equation*}
       i_K((n,h_1),h_2) = \delta_N^{(2)}(n)i_H(h_1,h_2)
 \end{equation*}
 for all $n \in N$ and $h_1,h_2 \in H$.
 This implies that the net of permutation characters $\phi_j \cdot i_H \in \Ch(F \times H \times H)$ converges to $\infl^{F\times H \times H}_{K\times H}(i_K) = \phi \cdot i_H$.
 We deduce that $K = N \times H$ is sofic relative to $H$.
\end{proof}

\section{Main results and applications}\label{sec:mainResults}

In this final section we will put everything together to prove our main results and give some applications.
We begin with the definition of the central notion: the $L^2$-multiplicity.
Let $G$ be a group and let $H \leq G$ be a finite subgroup and let
 \begin{equation*}
     Q_{\bullet}\colon \dots \longrightarrow Q_{n+1} \stackrel{\partial_{n+1}}{\longrightarrow} Q_{n} \stackrel{\partial_{n}}{\longrightarrow} Q_{n-1} \longrightarrow \dots
 \end{equation*}
 be a chain complex of f.g.\ projective $\bbC[G]$-modules.

\begin{definition}\label{def:L2-multiplicity}
 Let $\chi \in \Irr(H)$ and $p \in \bbZ$.
 The \emph{$p$-th $L^2$-multiplicity of $\chi$ in $Q$} is
 \begin{equation*}
      m^{(2)}_p(\chi, Q_{\bullet}) := \frac{\chi(1)}{|H|} \: b^\phi_p( Q_{\bullet} )
 \end{equation*}
 where $b^\phi_p( Q_{\bullet} )$ is the $p$-th Betti number (see Section~\ref{sec:BettiNumbers}) w.r.t.\ $\phi = \ind_H^G(\widetilde{\chi})$
 and  $\widetilde{\chi} = \frac{\chi}{\chi(1)}$ is the normalized character associated to $\chi$.
\end{definition}

\subsection{$L^2$-multiplicities for finite groups}
We first discuss the $L^2$-multi\-plicities in the case of finite groups and show that they are the normalized ordinary multiplicities.
Before that, we recall how the ordinary multiplicities are defined.

Let $G$ be a finite group.
For finite dimensional right $\bbC[G]$-modules $W$ and $U$,
the \emph{multiplicity of $W$ in $U$}
is
\begin{equation*}
   m(W, U) = \dim_\bbC \Hom_{\bbC[G]}(U,W).
\end{equation*}
It is useful to keep in mind that $m(U,W) = m(W,U)$.
If $\vartheta$ is the ordinary character afforded by $W$, then we write $m(\vartheta,U) := m(W,U)$.

\medskip

Let $\phi \in \Ch(G)$ be a character and
let $A \in M_{m,n}(\bbC[G])$. First of all, we need to understand what $\rnk_\phi(A)$ and $\nulli_\phi(A)$ are.
It is sufficient to consider the rank (by \eqref{eq:dim-formula}).
By Corollary \ref{cor:PropertiesOfRank} the rank and nullity are affine in the character $\phi$. Further, since $\Ch(G)$ is the simplex over
the normalized irreducible characters, we may assume that $\phi = \widetilde{\vartheta}$ with $\vartheta \in \Irr(G)$.

Let $(V,\rho)$ be an irreducible representation which affords the character $\vartheta$. As we have seen in Section \ref{sec:finitedimensional}
the tracial Hilbert $G$-bimodule $\calH_\phi$ is $\End_\bbC(V)$ with the left and right multiplication by the elements of $\rho(G)$.
The inner product is given by $\langle a , b \rangle = \frac{1}{\vartheta(1)}\Tr(b^*a)$ and the trace-element is $\id_V$.

Consider the left multiplication map $l_A\colon \bbC[G]^n \to \bbC[G]^m$. Every $\bbC[G]$-module is projective and we obtain 
a decomposition $\bbC[G]^m = \im(l_A) \oplus C$. Therefore the image of $l^\phi_A \colon \calH_\phi^n \to \calH^m_\phi$
is isomorphic to $\im(l_A) \otimes_{\bbC[G]} \calH_\phi$.
It suffices to understand the dimension $\dim_{\calN_\phi}(W \otimes_{\bbC[G]} \calH_\phi)$ for any finite dimensional right $\bbC[G]$-module $W$.
The additivity of the dimension allows us to restrict to the case where $W$ is irreducible.

Under these assumption we have $W \otimes_{\bbC[G]} \calH_\phi$ is isomorphic to $W$ or is trivial depending on whether or not $W$ has character $\vartheta$.
Assume that $W$ has character $\vartheta$
and identify $\calH_\phi$ with a matrix algebra $M_d(\bbC)$ by choosing a basis of $V$ (where $d = \vartheta(1)$).
Then $W$ is isomorphic to the submodule given by the first row and the orthogonal projection is given by the matrix 
\begin{equation*}
   P = \begin{pmatrix}
           1 &  &  &  &\\
            & 0 &  &  &\\
            &  & \ddots &\\
            & & & 0\\
       \end{pmatrix}.
\end{equation*}
We conclude that $\dim_{\calN_\phi}(W \otimes \calH_\phi) = \frac{1}{\vartheta(1)} \Tr(P) = \frac{1}{\vartheta(1)}$.

Finally, putting these observations together we deduce that the rank $\rnk_\phi(A)$ equals $\frac{1}{\vartheta(1)} m(\vartheta, \im(l_A))$.
The following lemma summarizes this discussion.
\begin{lemma}\label{lem:multiplicitiesFiniteGroup}
 Let $\widetilde{\vartheta} = \frac{\vartheta}{\vartheta(1)}$ be the normalization of an ordinary character.
 For every chain complex $Q_{\bullet}$ of finite dimensional right $\bbC[G]$-modules we have
 \begin{equation*}
    b^{\widetilde{\vartheta}}_p(Q_{\bullet}) = \frac{1}{\vartheta(1)} \: m\left(\vartheta,H_p(Q_{\bullet})\right). 
 \end{equation*}
\end{lemma}
\begin{corollary}\label{cor:FiniteGroupMultiplicity}
 Let $G$ be a finite group and let  $H$ be a subgroup.
 For every irreducible character $\chi \in \Irr(H)$ the $L^2$-multiplicity is
  \begin{equation*}
    m^{(2)}_p(\chi, Q_{\bullet}) = \frac{m\left(\chi,H_p(Q_{\bullet})_{|H}\right)}{|G|}. 
 \end{equation*}
\end{corollary}
\begin{proof}
 The character $\ind_H^G(\widetilde{\chi})$ is the normalization of the ordinary induced character $\vartheta = \Ind_H^G(\chi)$.
 Note that $\vartheta(1) = \chi(1) [G:H]$.
 We apply Lemma \ref{lem:multiplicitiesFiniteGroup} and Frobenius reciprocity (Lemma \ref{lem:FrobeniusReciprocity}) 
 to obtain
 \begin{equation*}
   b^{\widetilde{\vartheta}}_p(Q_{\bullet}) = \frac{1}{\vartheta(1)} \:m\left(\vartheta,H_p(Q_{\bullet})\right) = \frac{m\left(\chi,H_p(Q_{\bullet})_{|H}\right)}{\chi(1)[G:H]}.
 \end{equation*}
\end{proof}

\subsection{$L^2$-multiplicities and $L^2$-Betti numbers}
Let $G$ be a group and let $H \leq G$ be a finite subgroup and let $Q_\bullet$
be a chain complex of f.g.\ projective $\bbC[G]$-modules.

\begin{lemma}
 The $p$-th $L^2$-multiplicities and the $p$-th $L^2$-Betti number of $Q_\bullet$ satisfy the relation
 \begin{equation*}
     \sum_{\chi \in \Irr(H)} \chi(1) \: m_p^{(2)}(\chi,Q_\bullet) = b^{(2)}_p(Q_\bullet).
 \end{equation*}
\end{lemma}
\begin{proof}
 We have the following identity of characters 
 \begin{equation*}
    \sum_{\chi \in \Irr(H)} \frac{\chi(1)^2}{|H|} \:  \ind_H^G(\widetilde{\chi}) =
    \sum_{h\in H} i_G( \cdot ,h) \sum_{\chi \in \Irr(H)} \frac{\chi(1)}{|H|} \chi(h) = \ind_H^G(\delta_H^{(2)}) = \delta_G^{(2)}.
 \end{equation*}
 The Betti numbers $b_p^\phi(Q_\bullet)$ are affine in the parameter $\phi$ (see Corollary \ref{cor:PropertiesOfRank} \eqref{item:affine}) and so the
 following calculation verifies the assertion
 \begin{equation*}
     \sum_{\chi \in \Irr(H)} \chi(1) \: m_p^{(2)}(\chi,Q_\bullet) =  \sum_{\chi \in \Irr(H)}\frac{\chi(1)^2}{|H|} \: b^{\ind^G_H(\widetilde{\chi})}_p( Q_{\bullet} )= b^{(2)}_p(Q_\bullet).
 \end{equation*}
\end{proof}
 Now we prove Proposition \ref{prop:centralizer} stated in the introduction; this is our main tool to calculate the multiplicities.
\begin{proposition}\label{prop:centralizer}
  Assume that every non-trivial element $h\in H$ has a centralizer $C_G(h)$ of infinite index in $G$.
  Then
  \begin{equation*}
       m_p^{(2)}(\chi, Q_\bullet) = \frac{\chi(1)}{|H|} b^{(2)}_p(Q_\bullet).
  \end{equation*}
\end{proposition}
\begin{proof}
  Indeed, under the assumptions $i_G(g,h) = 0$ unless $g=1$ and $h=1$. Consequently,
  \begin{equation*}
     \ind_H^G(\widetilde{\chi})(g) = \sum_{h\in H} i_G(g,h) \widetilde{\chi}(h) = \delta^{(2)}_G(g)
  \end{equation*}
   for all $g\in G$. This character identity immediately finishes the proof.
  \end{proof}

\subsection{The sofic approximation theorem}
We will use the notion of $G$-CW-complexes as in the survey \cite[Sect.~1.2.1]{LuckBook} with the warning that here the group $G$ acts from the right.
To be more precise, a $G$-CW-complex is a CW-complex $X$ with a cellular \emph{right} action of $G$ such that
for every open cell $e \subseteq X$ if $e\cdot g \cap e \neq \emptyset$, then $g$ acts trivially on $e$.
We say that $X$ is \emph{proper}, if every stabilizer is a finite subgroup of $G$.

Let $X$ be a proper $G$-CW-complex. The rational cellular chain complex of $X$ will be denoted by $C_\bullet(X)$.
Every $C_p(X)$ is isomorphic to a direct sum
\begin{equation*}
    C_p(X) = \bigoplus_{i\in I_p} \bbQ[S_i\backslash G]
\end{equation*}
where $I_p$ a set of representatives of the $G$-orbits on $p$-cells and $S_i$ is the \emph{finite} stabilizer.
Suppose that $X/G$ is of finite type, i.e.\ has only finitely many cells of each dimension, then Lemma \ref{lem:finitestabilizersProjective}
shows that $C_\bullet(X)$ is a complex of finitely generated projective $\bbQ[G]$-modules.

Let $H \leq G$ be a finite subgroup and let $\chi\in \Irr(H)$ be an irreducible character.
Let $p$-th $L^2$-multiplicity of $\chi$ in $X$ is defined to be
\begin{equation*}
    m^{(2)}_p(\chi, X; G) :=  m^{(2)}_p(\chi, C_\bullet(X)).
\end{equation*}
\begin{theorem}\label{thm:eqSofic}
 Let $G$ be a group and let $H$ be a finite subgroup. 
 Suppose that $X$ is a proper $G$-CW-complex such that $X/G$ is of finite type.
 Let $\Gamma_1 \supset \Gamma_2 \supset \Gamma_3 \supset \dots$ be a decreasing sequence normal subgroups of $G$ such that
 $\bigcap_{n\in \bbN} \Gamma_n = \{1\}$. We can assume that $\Gamma_n \cap H = \{1\}$ for all $n$.
 
 If $G/\Gamma_n$ is sofic relative to $H\Gamma_n/\Gamma_n$ for all sufficiently large $n$, then the $p$-th $L^2$-multiplicities satisfy
 \begin{equation*}
    \lim_{n\to \infty} m^{(2)}_p(\chi, X/\Gamma_n; G/\Gamma_n) = m^{(2)}_p(\chi, X; G)
 \end{equation*}
 for every $\chi\in \Irr(H)$.
\end{theorem}
\begin{proof}
  First we observe that $H\Gamma_n/\Gamma_n \cong H/H\cap\Gamma_n \cong H$ and we will tacitly identify $H$ with $H\Gamma_n/\Gamma_n$.
  Let $q_n$ denote the canonical quotient morphism $G \to G/\Gamma_n$.
  
  Consider the rational cellular chain complex 
  \begin{equation*}
       \cdots \longrightarrow C_{p+1}(X) \stackrel{\partial_{p+1}}{\longrightarrow} C_{p}(X) \stackrel{\partial_{p}}{\longrightarrow} C_{p-1}(X) \longrightarrow \cdots
  \end{equation*}
   We apply Lemma \ref{lem:existenceOfMatrices} to find matrices $A$ and $B$ over the rational group ring $\bbQ[G]$ such that
   $\rnk_\eta(A) = \rnk_\eta(\partial_{p+1})$ and $\nulli_\eta(B) = \nulli_\eta(\partial_p)$ for all $\eta \in \Ch(G)$.
   
   We introduce the following notation: $\psi_n = \ind_H^{G/\Gamma_n}(\widetilde{\chi})$, $\phi_n = q_n^*(\psi_n)$ and
   $\phi = \ind_H^{G}(\widetilde{\chi})$. In this notation we obtain
   \begin{align*}
      m_p^{(2)}(\chi,X/\Gamma_n;G/\Gamma_n) = &\frac{\chi(1)}{|H|} b^{\psi_n}(C_\bullet(X/\Gamma_n)) \\
      = &\frac{\chi(1)}{|H|} (\nulli_{\psi_n}(q_n(B)) - \rnk_{\psi_n}(q_n(A)))\\
      \stackrel{\text{Cor.\ref{cor:PropertiesOfRank}}}{=} &\frac{\chi(1)}{|H|} (\nulli_{\phi_n}(B) - \rnk_{\phi_n}(A)) = \frac{\chi(1)}{|H|} b^{\phi_n}(C_\bullet(X)).
   \end{align*}
   Note that $\lim_{n\to\infty} \phi_n = \phi$ by Lemma \ref{lem:convergenceInducedCharacters}. We deduce that 
   it suffices to show that the sequence $\phi_n$ has the approximation property with
   respect to all matrices over the rational group ring. We will verify that the characters $\phi_n$ are arithmetically hyperlinear of bounded degree,
   then the approximation property follows from Theorem \ref{thm:approxArith}.
   
   The assumption that $G/\Gamma_n$ is sofic relative to $H$ yields in particular that $i_G \in \Ch(G\times H)$ is arithmetically hyperlinear of degree at most $1$ relative to $H$.
   Theorem \ref{thm:inductionViaArithmeticallyHyperlinear} now shows that the characters $\psi_n = \ind_H^{G/\Gamma_n}(\widetilde{\chi})$ are arithmetically hyperlinear of 
   degree at most $\varphi(|H|)$.
\end{proof}
\begin{remark}
 Theorem \ref{thm:eqSofic} together with Corollary \ref{cor:FiniteGroupMultiplicity} provides the first proof of Theorem \ref{thm:approximation1} from the introduction.
\end{remark}

\subsection{A Farber approximation theorem for $L^2$-multiplicities}
 Let $X$ be a $G$-CW-complex such that $X/G$ is of finite type.
 In this section we discuss an approximation theorem for sequences $(\Gamma_n)_{n\in \bbN}$ 
 of finite index subgroups of $G$ which are not necessarily normal. Nevertheless, we need the finite group $H\leq G$ to act on 
 the quotient space $X/\Gamma_n$, hence we always assume that $H$ normalizes the groups~$\Gamma_n$.
 In this case $H$ acts not only on the quotient $X/\Gamma_n$ but, moreover,
 acts on the homology groups $H_p(X/\Gamma_n,\bbC)$. In particular, we can consider the multiplicities
 $m(\chi, H_p(X/\Gamma_n,\bbC))$ for every character $\chi \in \Irr(H)$.

 Our approach is inspired by the approximation theorem of Farber \cite{Farber1998} which is based on so-called \emph{Farber sequences}.
 A sequence $(\Gamma_n)_{n\in \bbN}$ of finite index subgroups of $G$ is called \emph{Farber},
 if the sequence of characters $\phi_n$ associated to the permutation representations of $G$ on $G/\Gamma_n$
 converges to the regular character $\delta_G^{(2)}$.
 One can make this explicit and obtains the following condition: A sequence $(\Gamma_n)_{n\in \bbN}$ is Farber if
 \begin{equation}\tag{Fa}\label{eq:Farber}
    \lim_{n\to \infty} \frac{\left|\{\:f \in G/\Gamma_n \:|\: g \in f\Gamma_nf^{-1} \:\}\right|}{[G:\Gamma_n]} = 0
 \end{equation}
 for all $g \in G$ with $g\neq 1$.
 Theorem 0.3 of Farber in \cite{Farber1998} generalizes L\"uck's approximation theorem to all Farber sequences of finite index subgroups.
 
 Here we will obtain a ``Farber''-version of the approximation theorem for $L^2$-multiplicities.
 However, it turns out the Farber condition is not sufficient for our purposes and we need to take into account the finite group $H$.
 \begin{definition}
   Let $G$ be a group and let $H$ be a finite subgroup.
   A sequence $(\Gamma_n)_{n\in \bbN}$ of finite index subgroups of $G$ is
   \emph{Farber relative to $H$} if $H$ normalizes each $\Gamma_n$ and
   the sequence of characters of the $G$-$H$-bisets $G/\Gamma_n$ converges to $i_G$. This means,
   \begin{equation}\tag{Fa/$H$}\label{eq:relFarber}
      \lim_{n\to\infty} \frac{\left|\{\:f \in G/\Gamma_n\:|\:f^{-1}gf \in h\Gamma_n\:\}\right|}{[G:\Gamma_n]} = i_G(g,h)
   \end{equation}
   for all $g\in G$ and $h\in H$.
 \end{definition}
  \begin{remark}
   Let $\Gamma_1 \supset \Gamma_2 \supset \Gamma_3 \supset \dots$ be a decreasing sequence of finite index \emph{normal} subgroups.
   If $\bigcap_{n\in\bbN}\Gamma_n = \{1\}$, then the sequence is Farber relative to any finite subgroup $H\leq G$.
   Indeed, this follows from the proof of Lemma \ref{lem:convergenceInducedCharacters}.
  \end{remark}
  
  It is an easy observation that every Farber sequence relative to $H$ satisfies the Farber condition \eqref{eq:Farber} (put $h=1$!).
  The converse however is not true even if every $\Gamma_n$ is normalized by $H$. 
  \begin{example}
     Let $D_\infty = \bbZ \rtimes \{\pm 1\}$ be the infinite dihedral group with the finite subgroup $H = 0 \rtimes \{\pm 1\}$.
     We consider the sequence of finite index subgroups $\Gamma_m = m\bbZ \rtimes \{\pm 1\}$.
     Clearly, every $\Gamma_m$ is normalized by $H$; however, the $\Gamma_m$ are not normal (unless $m=2$).
     The sequence is not Farber relative to $H$. Indeed, since $H\subseteq \Gamma_m$ for all $m$ we can take
     the element $g=1$ and $h=(0,-1)$ to see that \eqref{eq:relFarber} is violated.
     
     Finally, we claim that the sequence $\Gamma_m$ is Farber.
     To this end we verify that 
     \begin{equation*}
        \left|\{\:f \in D_\infty/\Gamma_m\:|\:f^{-1}(k,\varepsilon)f \in \Gamma_n\:\}\right| = \begin{cases}
                                                                                                 m \quad& \text{ if $k\equiv 0\pmod m$ and $\varepsilon=1$}\\
                                                                                                 \gcd(2,m) & \text{ if $k\in \gcd(2,m)\bbZ$ and $\varepsilon=-1$}\\
                                                                                                 0 & \text{ otherwise. }
                                                                                         \end{cases}
     \end{equation*}
     for all $k\in\bbZ$ and $\varepsilon\in \{\pm1\}$.
     Indeed, for $f=(i,1)$ we have $(-i,1)\cdot(k,\varepsilon)\cdot(i,1) = (k + (\varepsilon-1)i,\varepsilon)$;
     this is an element of $\Gamma_m$ if and only if $k \equiv i(1-\varepsilon) \pmod m$.
     The claim follows by considering the cases $(1-\varepsilon)$ equals $0$ or $2$ respectively.
  \end{example}
  
  Having emphasized this difference, we are now in the position to state and prove the approximation theorem subject to the relative Farber condition.
  
  \begin{theorem}\label{thm:eqFarber} Let $G$ be a group and let $H\leq G$ be a finite subgroup.
   Let $X$ be a proper $G$-CW-complex such that $X/G$ is of finite type.
   Suppose that $(\Gamma_n)_{n\in \bbN}$ is a sequence of finite index subgroups which is Farber relative to $H$ (see \eqref{eq:relFarber}),
   then
   \begin{equation*}
       \lim_{n\to\infty} \frac{m(\chi, H_p(X/\Gamma_n,\bbC))}{[G:\Gamma_n]} = m_p^{(2)}(\chi, X; G)
   \end{equation*}
   for all $\chi \in \Irr(H)$.
  \end{theorem}
  \begin{proof}
    Let $\psi_n$ be the character of the $G$-$H$-biset $Y_n = G/\Gamma_n$; this means,
    $\psi_n(g,h) = \frac{|\{y \in Y_n\:|\: gyh^{-1} = y\}|}{|Y_n|}$.
    The sequence $(\Gamma_n)_{n\in\bbN}$ being Farber relative to $H$, means that $\lim_{n\to\infty}\psi_n = i_G \in \Ch(G\times H)$.
    In particular, we have
    \begin{equation*}
       \lim_{n\to\infty} \sum_{h\in H} \psi_n(1,h) \widetilde{\chi}(h) = \sum_{h\in H} i_G(1,h) \widetilde{\chi}(h) = 1.
    \end{equation*}
    This shows that the normalized character $\widetilde{\chi}$ can be induced via $\psi_n$ (in the sense of Definition \ref{deflem:induction}) for all sufficiently large $n\in \bbN$.
    Let $\phi_n = \ind_{\psi_n}(\widetilde{\chi})$ be the corresponding induced character.
    We note (using $\psi_n \to i_G$) that the sequence of characters $(\phi_n)_{n\in\bbN}$ converges to the induced character $\ind_H^G(\widetilde{\chi})$.
    
    By Theorem \ref{thm:inductionViaArithmeticallyHyperlinear} the characters $\phi_n$ are arithmetically hyperlinear of degree at most $\varphi(|H|)$. 
    We deduce, using Theorem \ref{thm:approxArith}, that the sequence $(\phi_n)_{n\in\bbN}$ has the approximation property (Def.~\ref{def:approximationProperty}) with respect to all matrices over the rational
    group ring $\bbQ[G]$. Notably, we obtain
    \begin{equation}\label{eq:reductionToBettiNumbers}
        \lim_{n\to\infty} b_p^{\phi_n}(C_\bullet(X)) = b^{\ind_H^G(\widetilde{\chi})}_p(C_\bullet(X)) = \frac{|H|}{\chi(1)} m_p^{(2)}(\chi,X; G).
    \end{equation}
    
    Fix $n\in \bbN$ and put $\phi := \phi_n$ and $\psi:= \psi_n$.
    As a next step we want to understand how the Betti numbers $b_p^{\phi}(C_\bullet(X))$ are related to the multiplicities.
    This reduces to a statement about finite groups.
    Let $N$ be the largest normal subgroup of $G$ which is contained in $\Gamma_n$, that is, $N = \bigcap_{g\in G} g\Gamma_n g^{-1}$.
    We will consider everything reduced modulo $N$.
    Let $\psi'$ be the character of $Y_n$ considered as a $G/N$-$H$-biset, i.e.\ $\infl^G_{G/N}(\psi') = \psi$. And similarly, let $\phi'= \ind_{\psi'}(\widetilde{\chi})$ so that
    $\phi = \infl^G_{G/N}(\phi')$.
    Now we note that $b_p^{\phi}(C_\bullet(X)) = b_p^{\phi'}(C_\bullet(X/N))$ (this follows from Corollary \ref{cor:PropertiesOfRank}).
    We may thus assume that $G$ is a finite group, and that $X$ is a finite type $G$-CW-complex for the calculation of $b_p^{\phi}(C_\bullet(X))$.
    
    Let $M = \bbC[Y_n]$ be the $\bbC[G]$-$\bbC[H]$-bimodule associated with $Y_n$ and let $V$ be a right $\bbC[H]$-module which affords $\chi$.
    Let $\vartheta$ be the ordinary character of the induced representation $\Ind_M(V)$.
    By Lemma \ref{lem:inducedCharacterFormula} the identity $\vartheta = \frac{\chi(1)}{|H|} \phi$ holds, i.e.\ $\widetilde{\vartheta} = \phi$.
    We deduce using Lemma \ref{lem:multiplicitiesFiniteGroup} and Lemma \ref{lem:FrobeniusReciprocity} that
    \begin{equation*}
       b_p^\phi(C_\bullet(X)) = \frac{1}{\vartheta(1)} m(\Ind_M(V), H_p(X,\bbC)) = \frac{1}{\vartheta(1)} m(\chi, \Res_M(H_p(X,\bbC))).
    \end{equation*}
    
    We conclude the proof with three observations.
    First of all, it was pointed out in Example \ref{ex:generalizedRestriction} that $\Res_M(H_p(X,\bbC))$ is simply the module
    $H_p(X,\bbC)_{\Gamma_n}$ of $\Gamma_n$-coinvariants of $H_p(X,\bbC)$.
    Since we are working with finite groups, every $\bbC[\Gamma_n]$-module is projective and, in addition, taking coinvariants is an exact functor.
    We conclude that $H_p(X,\bbC)_{\Gamma_n} = H_p(C_\bullet(X,\bbC)_{\Gamma_n}) = H_p(X/\Gamma_n,\bbC)$.

    The second observation is that
    \begin{equation*}
    \vartheta(1) = \dim_\bbC(\Ind_M(V)) = |Y_n/H| \dim_\bbC(V^{H\cap \Gamma_n}). 
    \end{equation*}
    We  briefly explain this identity. Recall that $\Ind_M(V) = \Hom_{\bbC[H]}(\bbC[Y_n], V)$.
    A homomorphism $\alpha$ maps every basis element $y \in Y_n$ an element in $\alpha(y) \in V$ in an $H$-equivariant way.
    The homomorphism $\alpha$ is completely determined if we know its value at a representative of every (right) $H$-orbit in $Y_n$.
    However, the value cannot be arbitrary; there is one restriction: the value $\alpha(y)$ has to be invariant under the stabilizer of
    $y$ in $H$. Let $y = f\Gamma_n$, then $yh = fh\Gamma_n \stackrel{?}{=} f\Gamma_n = y$ exactly if $h \in H\cap\Gamma_n$.
    
    The last remark is that $\Gamma_j \cap H = \{1\}$ for all sufficiently large $j$.
    In fact, for $h\in H$ with $h\neq 1$ the relative Farber condition yields
    \begin{equation*}
        \lim_{j\to\infty} \frac{|\{\:f \in G/\Gamma_j \:|\: 1 \in h\Gamma_j\:\}|}{[G:\Gamma_n]} = i_G(1,h) = 0.
    \end{equation*}
    This implies that $h\not\in \Gamma_j$ for all sufficiently large $j\in \bbN$.
    
    The proof of the Theorem is complete when we insert the formula
    \begin{equation*}
      b_p^{\phi_n}(C_\bullet(X)) = \frac{|H|}{\dim_\bbC(V^{H\cap \Gamma_n})} \: \frac{m(\chi, H_p(X/\Gamma_n,\bbC))}{|H\cap\Gamma_n| \cdot [G:\Gamma_n]}
    \end{equation*}
    in the equation \eqref{eq:reductionToBettiNumbers}. Note that $\dim_\bbC(V) = \chi(1)$.
  \end{proof}

\subsection{Applications to the cohomology of groups}
  In this section we discuss applications of our main theorem to the cohomology of groups.
  It seems to us that these observations are not obvious to obtain without $L^2$-methods.
  
  Recall that a group $\Gamma$ is said to be of type $F_m$ if there is a $K(\Gamma,1)$-space with finite $m$-skeleton.
  Using this we derive the simple translation stated in the introduction.
  \begin{theorem}\label{thm:groupCohomology}
  Fix $p\in \bbN$.
   Let $\Gamma$ be a group of type $F_{p+1}$ and let $H$ be a finite group which acts on $\Gamma$ by automorphisms.
   Assume that the fixed point group $\Gamma^h$ has infinite index in $\Gamma$ for all $h\in H\setminus\{1\}$.
   
   Let $\Gamma \supset \Gamma_1 \supset \Gamma_2 \supset \dots$ be a decreasing sequence of finite index normal $H$-stable subgroups with trivial intersection.
    Then
   \begin{equation*}
       \lim_{n\to \infty} \frac{m(\chi, H^p(\Gamma_n,\bbC))}{[\Gamma:\Gamma_n]} = \frac{\chi(1)}{|H|} \: b^{(2)}_p(\Gamma) 
   \end{equation*}
   for every character $\chi \in \Irr(H)$.
  \end{theorem}
  \begin{proof}
   Consider the semidirect product group $G = \Gamma \rtimes H$. We will identify $\Gamma$ with $\Gamma \rtimes\{1\} \subset G$.
   Since $G$ contains $\Gamma$ as a finite index subgroup, the group $G$ has type $F_{p+1}$ as well
   (see \cite[7.2.3]{Geoghegan2008}).
   Let $X$ be a classifying space for $G$. This means, $X$ is a contractible free $G$-CW-complex. Since $G$ is of type $F_{p+1}$, we can assume
    that the quotient $X^{p+1}/G$ of the $(p+1)$-skeleton $X^{p+1}$ is of finite type.
   Note that the space $X^{p+1}/\Gamma_n$ is the $p+1$-skeleton of the Eilenberg-MacLane space of $\Gamma_n$ -- in particular,
   its (co)homology is isomorphic to the group (co)homology of $\Gamma_n$ up to degree $p$.
   Moreover, the inherited action of $H$ yields the action of $H$ in the cohomology of the group $\Gamma_n$.
   
   Our assumption that $\Gamma_n$ is $H$-stable implies that $\Gamma_n$ is a normal subgroup of $G$.
   Further, the centralizer of every non-trivial $h\in H$ inside $G$ is $C_G(h) = \Gamma^h \times C_H(h)$.
   We deduce that $C_G(h)$ has infinite index in $G$.
   Proposition \ref{prop:centralizer} now shows that
   \begin{equation*}
         m_p^{(2)}(\chi, X^{p+1}; G) = \frac{\chi(1)}{|H|} \: b_p^{(2)}(G) = \frac{\chi(1)}{|H|^2} \:  b_p^{(2)}(\Gamma)
   \end{equation*}
   where we use assertion (9) of Thm.\ 1.35 in \cite{LuckBook} in the last step.
   
   Now we apply Theorem \ref{thm:approximation1} to deduce
   \begin{equation*}
       \lim_{n\to \infty} \frac{m(\chi, H_p(X^{p+1}/\Gamma_n,\bbC))}{[G:\Gamma_n]} = \frac{\chi(1)}{|H|^2} \:  b_p^{(2)}(\Gamma).
   \end{equation*}
   We can cancel an appearance of $|H|$ in the denominator using $[G:\Gamma_n] = |H| [\Gamma:\Gamma_n]$.
   Finally, the universal coefficient theorem implies that the representation of $H$ in the $p$-th cohomology 
   $H^p(\Gamma_n,\bbC)$ is dual to the representation in the $p$-th homology group. However, the limit only depends on the 
   dimension of the irreducible representation $\chi$, and this property does not change under taking duals.
  \end{proof}
  
  \begin{corollary}\label{cor:traces}
    Let $\Gamma$ be a group of type $F_{p+1}$ and let $\sigma \in \Aut(\Gamma)$ denote an automorphism of finite order
    such that the fixed point group $\Gamma^{\sigma^j}$ has infinite index in $\Gamma$ for all $0 < j < \ord(\sigma)$.
    For every decreasing sequence of finite index normal subgroups $(\Gamma_n)_{n\in\bbN}$ with 
    $\bigcap_{n\in \bbN} \Gamma_n = \{1\}$, the normalized traces satisfy
    \begin{equation*}
       \lim_{n\to\infty} \frac{\Tr( \sigma \:|\: H^p(\Gamma_n,\bbC))}{[\Gamma:\Gamma_n]} = 0.
    \end{equation*}
  \end{corollary}
  \begin{remark}
    Assume that $\Gamma$ is of type $F$, i.e.\ has a finite $K(\Gamma,1)$-space.
    Then the Lefschetz number of $\sigma$ (the alternating sum of the traces) is well-defined.
    An immediate corollary is that the sequence of normalized Lefschetz numbers converges to zero.
    This consequence can indeed be proven in a more elementary way without any reference to $L^2$-methods.
  \end{remark}
  \begin{proof}[Proof of Corollary \ref{cor:traces}]
   Let $H = \langle \sigma \rangle$ be the cyclic group generated by $\sigma$.
   The hypotheses of Theorem~\ref{thm:groupCohomology} are satisfied.
   Note that the irreducible representations of $H$ are one-dimensional.
   The following calculation uses the orthogonality relations of characters of finite groups in the last step
   \begin{align*}
      \lim_{n\to\infty} \frac{\Tr( \sigma \:|\: H^p(\Gamma_n,\bbC))}{[\Gamma:\Gamma_n]}
      &= \lim_{n \to \infty} \sum_{\chi \in \Irr(H)} \chi(\sigma) \frac{m(\chi,H^p(\Gamma_n,\bbC))}{[\Gamma:\Gamma_n]}\\
      &= \frac{b_p^{(2)}(\Gamma)}{\ord(\sigma)} \: \sum_{\chi \in \Irr(H)} \chi(\sigma) = 0.
   \end{align*}
    This completes the proof.
  \end{proof}

  \subsection{Epilogue: An application to the cohomology of arithmetic groups}
  
  Let $F/\bbQ$ be a totally real Galois extension of degree $d = [F:\bbQ]$ with Galois group $\Gal(F/\bbQ)$.
  Let $\mathbb{G}$ be a semi-simple linear algebraic group defined over $\bbQ$ such that the group of real points 
   $G_\infty = \mathbb{G}(\bbR)$ is non-compact. We choose a maximal compact subgroup $K \leq G_\infty$.
  Recall that the fundamental rank of a semi-simple real Lie group $G_\infty$ is 
  the difference $\rnk_\bbC(G_\infty) - \rnk_\bbC(K)$ of its complex rank and the complex rank of $K$.
  The associated Riemannian symmetric space $K\backslash G_\infty$ will be denoted by $X_0$
  
  Extending scalars to $F$ we obtain a semi-simple linear algebraic group $\mathbb{G}\times_\bbQ F$ over $F$.
  The Galois group $\Gal(F/\bbQ)$ acts on $\mathbb{G}(F)$. In fact, applying Weil's restriction of scalars functor, this action can be
  interpreted as an action by rational automorphism.
  
  We are interested in arithmetic subgroups $\Gamma \subseteq \mathbb{G}(F)$ of the group $\mathbb{G}\times_\bbQ F$.
  Such an arithmetic subgroup acts on the associated locally symmetric space $X = X_0^d$.
  The dimension of this space is $\dim(X) = d (\dim(G_\infty) - \dim(K))$.
  If $\Gamma$ is stable under the Galois group $\Gal(F/\bbQ)$, then $\Gal(F/\bbQ)$ acts on the cohomology of $\Gamma$ and we are interested in this action.
  We obtain the following result.
\begin{corollary}\label{cor:GaloisAction}
  Suppose that $G_\infty$ has fundamental rank $0$ and let $p = \frac{1}{2}\dim(X)$ be the middle dimension of $X$.
   Let $\Gamma\subseteq \mathbb{G}(F)$ be an arithmetic subgroup of  which is $\Gal(F/\bbQ)$-stable
   and let $(\Gamma_n)_{n\in\bbN}$ be a decreasing sequence 
   of finite index normal $\Gal(F/\bbQ)$-stable subgroups with $\bigcap_{n} \Gamma_n =\{1\}$.
   For every irreducible representation $\vartheta\in\Irr(\Gal(F/\bbQ))$ the multiplicity of $\vartheta$ in degree $p$ satisfies
 \begin{equation*}
    \lim_{n\to \infty} \frac{m( \vartheta, H^p(\Gamma_n,\bbC) )}{[\Gamma:\Gamma_n]} = \frac{ (-1)^p \: \vartheta(1) \:\chi(\Gamma)}{[F:\bbQ]} 
 \end{equation*}
 where  $\chi(\Gamma)$ denotes the Euler characteristic of $\Gamma$ (in the sense of Wall).
 
 Since $\chi(\Gamma) \neq 0$, every irreducible representation of the Galois group occurs with arbitrarily large multiplicity in the middle cohomology
 of some $\Gamma_n$.
\end{corollary}
\begin{proof}
   The existence of the Borel-Serre compactification  shows that the group $\Gamma$ is of type $F_\infty$.
    To apply Theorem \ref{thm:groupCohomology}, we need to show that the fixed point group $\Gamma^\sigma$ has infinite index 
   in $\Gamma$ for every non-trivial $\sigma \in \Gal(F/\bbQ)$.
   Note that the Galois group $\Gal(F/\bbQ)$ acts faithfully on $\mathbb{G}(\bbR \otimes_\bbQ F) = G_\infty^d$ by permuting the factors
   and $\Gamma$ is stable under this action.
   The next lemma (\ref{lem:nolattice}) shows that $\Gamma^\sigma$ cannot be a lattice in $G^d_\infty$; hence $[\Gamma:\Gamma^\sigma] = \infty$.
   
   The $L^2$-Betti numbers of the arithmetic group $\Gamma$ are well-known: we have $b^{(2)}_m(\Gamma) = 0$ unless $m = p$; in this case
   $b^{(2)}_p(\Gamma) = (-1)^p \chi(\Gamma) > 0$. One way to see this (pointed out to us by J.~Raimbault and R.~Sauer)
   is to use Gaboriau's proportionality principle \cite[Cor.~0.2]{Gaboriau2002} and
   the fact that the result is known for all cocompact lattices, which goes back to the work of Borel \cite{Borel1985} (see also \cite[Thm.~5.12]{LuckBook}).
   Another way to calculate the $L^2$-Betti numbers is to use the theory of limit multiplicities (as developed in \cite{deGeorgeWallach1978, RohlfsSpeh1987, Savin1989}) 
   and L\"uck's approximation theorem.
\end{proof}
\begin{remark}
  Corollary \ref{cor:introSL2} from the introduction is a direct consequence of this result, since $\SL_2(\bbR)$ has fundamental rank $0$.
\end{remark}
\begin{lemma}\label{lem:nolattice}
  Let $G_\infty$ be a non-compact, locally compact, unimodular, Hausdorff topological group.
  Let the symmetric group $S_n$ act on $G_\infty^n$ by permutation of the factors.
  Assume that $\Gamma\subseteq G_\infty^n$ is a discrete subgroup which is pointwise fixed under some non-trivial $\sigma\in S_n$,
  then $\Gamma$ is not a lattice in $G_\infty^n$ (i.e.\ has infinite invariant covolume).
\end{lemma}
\begin{proof}
 Looking at the cycle decomposition of $\sigma$, we may assume that $\sigma$ has exactly one non-trivial cycle. 
 For simplicity, we ignore the fixed points and we assume that $\sigma = (1 \: 2 \: \dots \: n)$ is an $n$-cycle.
 In particular, there is a subgroup $\Gamma_0 \subseteq G_\infty$ such that $\Gamma$ is the image of $\Gamma_0$ under the diagonal embedding $G_\infty\to G_\infty^n$.
 
 Consider the continuous surjective map $\widetilde{p}: G_\infty^n \to G_\infty^{n-1}$ which takes $g = (g_1,\dots,g_n)$ to $\widetilde{p}(g) = (g_2g^{-1}_1,\dots g_n g_1^{-1})$.
 It factors to a well-defined continuous map $p: G_\infty^n/\Gamma \to G_\infty^{n-1}$.
 Let $\nu$ be a non-trivial $G^n$-invariant Radon measure on $G_\infty^n/\Gamma$. 
 Note that $G_\infty^n$ and $\Gamma$ are unimodular, so that such a measure exists by the work of Weil, see e.g.\ \cite[VIII 3.25]{Elstrodt2005}.
 We claim that the push-forward measure $p_*(\nu)$ is a left Haar measure on $G_\infty^{n-1}$. Indeed,
 \begin{align*}
     p_*(\nu((g_2,\dots g_n)\cdot X)) &= \nu\left(p^{-1}((g_2,\dots g_n)\cdot X)\right)\\
          &= \nu\left((1,g_2,\dots,g_n)\cdot p^{-1}(X)\right) = p_*(\nu)(X)
 \end{align*}
 for every measurable subset $X\subseteq G_\infty^{n-1}$.
 Note that $p_*(\nu)$ is non-trivial, since $\nu$ is non-trivial.
 However, since $G_\infty$ is non-compact it follows that $\nu(G_\infty^n/\Gamma) = p_*(\nu)(G_\infty^{n-1}) = \infty$
 by \cite[VIII 3.15]{Elstrodt2005}.
\end{proof}

\providecommand{\bysame}{\leavevmode\hbox to3em{\hrulefill}\thinspace}
\providecommand{\href}[2]{#2}

\end{document}